\documentclass[referee]{sn-jnl}

\usepackage{amsmath, array}
\usepackage{amsfonts,bbm}
\usepackage{amssymb}
\usepackage{amsthm}
\usepackage{multicol}
\usepackage{graphicx}
\usepackage{subfigure}
\usepackage{fancyhdr}
\usepackage{cases}
\usepackage[ruled,vlined]{algorithm2e}
\usepackage{color, colortbl,xcolor,soul}
\definecolor{lightgreen}{rgb}{.7,95,.65}
\sethlcolor{lightgreen}
\usepackage{tikz}
\usepackage{multirow}%
\usepackage{manyfoot}%
\usepackage{hyperref}
\usepackage[export]{adjustbox}
\hypersetup{
    colorlinks=true
}

\usepackage{tikz}

\usepackage{biblatex}
\definecolor{bordeaux}{rgb}{0.75,0.0,0.0}
\newcommand{\bord}{\color{bordeaux}}
\newcommand\blue{\color{blue}}    
    
\def\RRR{\mathsf{R}}
\DeclareMathOperator*{\argmax}{\arg\!\max}
\setstcolor{blue}

\definecolor{mbrown}{cmyk}{0.1, 0.5, 0.9, 0.5}
\definecolor{mvert}{cmyk}{0.1,0.7, 0.6, 0.8}

\usepackage{fixme} 
\fxsetup{theme=color,mode=multiuser}
\FXRegisterAuthor{Carlotta}{CD}{\color{mbrown}\footnotesize\textbf{Carlotta}}
\FXRegisterAuthor{Debora}{DA}{\color{red}\footnotesize\textbf{Debora}}
\FXRegisterAuthor{Kwame}{KAG}{\color{mvert}\footnotesize \textbf{Kwame}}
\fxsetup{status=draft} 

\newtheorem{theorem}{Theorem}[section]

\newtheorem{lemma}[theorem]{Lemma}
\newtheorem{proposition}{Proposition}[section]
\newtheorem{definition}[theorem]{Definition}

\newtheorem{remark}{Remark}

\renewcommand{\theequation}{\thesection.\arabic{equation}}

\newcommand{\eN}{\mathbb{N}}

\newcommand{\eR}{\mathbb{R}}
\newcommand{\eZ}{\mathbb{Z}}

\newcommand{\Cp}[1]{\mathbf{C^{#1}}}
\newcommand{\Lp}[1]{\mathbf{L^{#1}}}

\DeclareMathOperator{\sign}{sign}

\newcommand{\norm}[1]{{\left\|#1\right\|}}

\newcommand{\fhi}{\varphi}

\newcommand{\pt}{\partial_t}

\newcommand{\px}{\partial_x }

\newcommand{\rL}{\rho_L}
\newcommand{\rR}{\rho_R}
\newcommand{\rM}{\rho_M}

\newcommand{\mG}{\mathcal{G}}
\newcommand{\LD}{L^1D}
\newcommand{\mF}{\mathcal{F}}
\newcommand{\DX}{\mathsf{h}}
\newcommand{\DT}{\mathsf{k}}
\newcommand{\CCC}{\mathsf{C}}
\newcommand{\xx}{\mathsf{x}}

\begin{document}

\title[Godunov scheme for discontinuous conservation law]{A Godunov--type scheme for a scalar conservation law with space-time flux discontinuity}

\author*[1,2]{\fnm{Kwame Atta} \sur{Gyamfi}}\email{kwame.gyamfi@ingv.it}




\affil[1]{ \orgname{Istituto Nazionale di Geofisica e Vulcanologia}, \orgaddress{\street{Roma1}, \city{L'Aquila}, \country{Italy}}}

\affil[2]{\orgdiv{Department of Mathematics}, \orgname{Kwame Nkrumah University of Science and Technology, Kumasi}, \country{Ghana}}


\equalcont{\small{\textbf{Acknowledgments:} The author acknowledges the support and guidance of Prof. Debora Amadori during their doctoral studies at the University of L'Aquila. Additionally, 
the author expresses gratitude to the research group of Prof. Carlotta Donadello at the University of Franche-Comt\'{e} Besan\c{c}on for their hospitality during research visits.}}

\abstract{
We present and analyze a new finite volume scheme of Gudonov-type for a nonlinear scalar conservation law whose flux function has a discontinuous coefficient due to time-dependent changes in its sign along a Lipschitz continuous curve. 
}




\maketitle

	\section{Introduction}\label{intro}
In this article, we propose a numerical scheme for solving a scalar conservation law with a flux function that is discontinuous in space and time. The equation writes 
	\begin{equation}\label{eq:main}
	\begin{cases}
		\pt\rho + \px F(t,x,\rho) = 0,\\
		\rho(0,x) = \rho_0(x),\;\rho(t,\pm 1) = 0,
	\end{cases}
	\end{equation}
	in which
	\begin{itemize}
		\item $t\in \eR^+ = [0,\infty)$ is the time variable and $x\in \Omega = ]-1,1[$ is the space variable; 
		 
		\item $(t,x)\mapsto\rho(t,x)$ is the unknown function, and $\rho\in\Lp\infty(\eR^+\times\Omega; [0,1])$;
			 
		\item $F$ representing the flux function is given by
				\begin{equation}\label{eq:def-of-F}
				F(t,x,\rho) \,\dot =\, \sign(x-\xi(t))f(\rho)	    
				\end{equation}
				where 
\begin{equation}\label{eq:hyp-on-f}
    f:[0,1] \to \eR^+ \mbox{ is 
    concave}\,, \quad f(0) = 0 = f(1)\,, \quad \max_{\rho\in [0,1]}{f(\rho)}>0\,,\quad 
    \lim_{\rho\to0+}\frac{f(\rho)}{\rho}\in\eR;
\end{equation}
\begin{equation}\label{eq:hyp-on-xi}
			     \xi:\eR^+\to \Omega \mbox{ is Lipschitz continuous with} 
       \quad 
			    \left[\inf_{\eR^+} \xi, \, \sup_{\eR^+} \xi\right] \subset (-1,1)\,.
			\end{equation}
	\end{itemize}

As equation \eqref{eq:main} is a fundamental law, most partial differential equations (PDEs) problems arising in the physical and engineering sciences can be formulated in its form. Examples can be found in (1) porous media: modeling the two-phase flow in porous media \cite{Gimse1992, Jaffre1995} and continuous sedimentation in ideal clarifier-thickener units \cite{BurgerR2004}; (2) traffic flow: modeling a toll gate along a highway \cite{Colombo2007} and traffic flow with changing road surface conditions \cite{RBurger2003traffic, Mochon1987traffic}; (3) ion etching in semiconductor device fabrication \cite{DSRoss1988}; (4) pedestrian flow models \cite{HUGHES2002507,DeboraGoatinRosini2014,El-Khatib2013Weak}.

More precisely, equation \eqref{eq:main} derives from the one-dimensional Hughes model \cite{HUGHES2002507} of pedestrian flow in a narrow corridor with fixed exits at $x= \pm 1$. However, in contrast to the Hughes model, the time-dependent flux interface $\xi(t)$ considered in this work is not required to satisfy the implicit relation involving a monotone cost function $c(\rho)$ 
\begin{equation}
    \label{eq:implicit-xi-Hughes}
    \int_{-1}^{\xi(t)}c(\rho(t,x))dx=\int_{\xi(t)}^{1}c(\rho(t,x))dx,\quad \forall~t \geq 0.
\end{equation}
Therefore, the equation considered in this work assumes no nonlocal conditions on the flux via the turning curve $x=\xi(t)$ even though the numerical scheme proposed here can easily be adapted to solve Hughes' model. It is worth noting that in the context of this model, the function $\xi(t)$ is commonly known as the turning curve and represents the points where pedestrians change direction to optimize the distance or time required to reach exits. See \cite{DeboraGoatinRosini2014}. In this paper, we will use the term "flux interface" or simply "the interface" to refer to $\xi$. Moreover, we will refer to a scalar conservation law with discontinuous flux as the discontinuous flux problem.

Research works dedicated to solutions to discontinuous flux problems have been extensively documented in the literature. See articles \cite{Jaffre1995,adimurthi2004godunov, dutta2014monotone,Towers2001,Bachmann2006,karlsen_risebro_towers,Andreianov2011,Diehl1995,Mishra2005, andreianov2012godunov,Xin_Shi2008, Bachmann2006}. It is a well-known fact that even if the flux function is continuous and the initial data do not contain jump discontinuities, the global Cauchy problem for a scalar conservation law does not admit classical solutions that are valid for all times $t> 0$. Thus, there exists a finite time beyond which classical solutions develop discontinuities, such as shock waves. The occurrence of these irregularities in the solution, which is severe for equations with discontinuous coefficients such as \eqref{eq:main}, is largely attributed to the nonlinearity of the flux function. For example, equation \eqref{eq:main} generalizes to the system
    \begin{equation}\label{eq:CL-systems}
        \pt\rho + \px(k(t,x)f(\rho)) = 0, \quad k_t = 0 \end{equation}
whose speed matrix (or Jacobian matrix) after writing \eqref{eq:CL-systems} in nonconservative form has a repeated zero eigenvalue for some values of $(\rho, k)$.
For this reason, the appropriate setting for an admissible solution is the space of discontinuous functions, such as the space of $ \Lp{\infty}$ functions, in which a shock wave is well-defined. Such solutions are sought in the weak (or distributional) sense and hence are referred to as weak solutions \cite{Evans:1529903}. Popular methods to construct solutions include numerical approximation methods based on finite volume/difference schemes \cite{Andreianov_Sylla2023FV, BurgerR2004, Karlsen_Towers2004, Karlsen_Towers2017}, wavefront tracking algorithm \cite{PaolaMimault2013, Gimse1992,Coclite_Risebro2005}, and others. In this paper, we discuss the existence of entropy weak solution to \eqref{eq:main} by constructing the approximate solution via a finite volume method exclusive to this class of equations.

In general, deriving an existence theory for a class of hyperbolic conservation law relies on {\em a priori} estimates with which a compactness argument can be used 'to pass to the limit' in the solution through an appropriate compactness theory such as Helly's theory. However, for selected time-dependent discontinuous flux problems, such as \eqref{eq:main}, such estimates are difficult to obtain or nonexistent, even after transforming approximate solutions via an appropriate singular mapping, as was done in \cite{Coclite_Risebro2005}. In such situations, alternative approaches that are free from these estimates, including compensated compactness (see \cite{Karlsen_Towers2004}), measure-valued solutions (see \cite{AndreianovGoatinSeguin2010}), and kinetic formulation (see \cite{dotti2019}), provide handy tools to establish rigorous existence results. In the latter two cases, it is enough to use only standard estimates, such as the so-called "weak BV estimates," to prove the convergence of the approximate solution to the entropy process solution under appropriate CFL conditions \cite{Eymard2000, eymard2007analysis}.

The notion of entropy solutions for discontinuous flux problems does not follow the standard notions suitable for a generic conservation law with continuous coefficients. The main ideas of entropy solutions for discontinuous flux problems rely on applying Kruzkhov's admissibility conditions away from the flux interface but enforcing additional "entropy conditions" at the interface to aid the selection of unique and admissible solutions. It has been argued in \cite{Mishra2005} that as a result of this occurrence, there exists an infinite number of $L^1$ contractive semigroups, each associated with an interface connection, for a variety of discontinuous flux problems. The principal features of entropy solutions are the shape of the flux and the traces in the solution at the interface (i.e., the so-called $A-B$ entropy connection) \cite{andreianov2015interface, dutta2014monotone, andreianov2012godunov}. In the work of \cite{Andreianov2011}, these concepts were unified by the introduction of a new framework called admissibility germs $\mG$ (or "germs" for short) within which the entropy solutions described above reduced to elementary solutions of piecewise constant nature provided their traces satisfy a form of a Rankine-Hugoniot condition at the interface. In our context, this family of piecewise constant constant solutions, typically described by ordered pairs $ (p_l, p_r)$ of the form $p_l\chi_{x<\xi(t)} + p_r\chi_{x>\xi(t)}$, with $p_l$ and $p_r$ in $[0,R]$ are maximal and complete.

In this context, we present and analyze a finite volume scheme to solve a scalar conservation law with a single time-dependent flux discontinuity that switches the sign of the flux coefficient between $+1$ and $-1$ at each $t>0$. We deal with the discontinuity by introducing a moving mesh near $\xi$ in the discrete space-time domain and then adapt the value of the flux there based on the slope $\dot{\xi}$. Standard finite volume schemes for continuous flux problems involve discretizing the space domain into numerical cells bounded by grid points. Then the solution for each discrete time is approximated by solving a 'local' Riemann problem at each grid point whose results are unionized to obtain the solution at the next time step. Here, the moving mesh consists in removing grid points nearest to $\xi$ so that the CFL condition is not violated. This approach introduced by \cite{ZHONG1996192}, has been used to study traffic flow models on a highway with a mobile bottleneck \cite{dellemonache2014, chalons2018conservative, delle2016numerical}. Indeed, the scheme proposed in these works was only analyzed in a mathematically rigorous in \cite{sylla2021lwr} where existence and uniqueness results were obtained in a fashion similar to ours. However, the scheme of \cite{sylla2021lwr} uses the Engquist-Osher flux on standard mesh away from the moving interface and the Godunov numerical fluxes at the interfaces with the moving mesh and a flux crossing condition. This scheme was later extended in \cite{Andreianov_Sylla2023FV}, where the existence of solutions for conservation laws with multiple space-time flux interface discontinuities. Indeed, any standard 'entropy' numerical flux such as the Godunovs, Lax Friederich's, etc, works well for treating solutions away from the interfaces as long as it is monotone and consistent.  We would like to highlight that the analysis conducted in \cite{Andreianov_Sylla2023FV} bears similarity to our own, and it may be considered more comprehensive, as it accounts for multiple moving interfaces. Nevertheless, considering the scarcity of contributions addressing space-time discontinuous flux problems, our work serves as an additional and valuable contribution to this field.

The remainder of the paper is organized into four main sections. Section \ref{sec:ProblemSettings} focuses on the theoretical aspect of the problem and discusses the Riemann problem for \eqref{eq:main}, including the Riemann solver at $\xi$. Here, we present the framework of admissibility germs and prove the relevant properties that could later be extended to obtain a comprehensive theory of the existence of weak entropy solutions in the next article. The next section, Section \ref{sec:FiniteVolumeScheme}, presents the numerical scheme and the mesh adaptations (the moving mesh), detailing all possible cases due to changes in slope of $\xi$. Furthermore, approximate entropy inequalities and a uniform $\Lp{\infty}$ bound on the numerical approximations are also presented. Finally, in Section \ref{sec:NumericalExamples}, we furnish details about the chosen examples, present their numerical simulations, and use them to explore the numerical convergence of the scheme achieved by calculating the error $\Lp{1}$, which is provided for each example.

\section{Problem setting}\label{sec:ProblemSettings}
We start this section by defining a weak entropy solution of \eqref{eq:main}, which will be featured throughout the paper. 
\begin{definition}\label{def:Kruzhov-solutions} 
Assume that $f$, $\xi$ satisfy \eqref{eq:hyp-on-f}, \eqref{eq:hyp-on-xi}, respectively. 
We say that map $(t,x)\mapsto\rho(t,x)$ is a weak entropy solution to the initial-boundary value problem \eqref{eq:main}, 
if $\rho$ is in $\Cp{0}\left([0,+\infty[; ~\Lp{{\blue\infty}}\left(\Omega; ~[0,1]\right)\right)$ and for any $c\in [0,1]$, 
and any test function $\varphi \in \mathbf{C}^{\infty}_{\mathbf{c}}\left(\eR^2; ~[0,+\infty[\right)$, 
the following Kru\v{z}khov-type entropy inequality holds:
	\begin{align} \label{eq:weak_solution}
	&\int_{0}^{+\infty}\!\!\int_{-1}^{1}\left[|\rho - c|\partial_t\varphi + \mathcal{F}(t,x,\rho,c)\partial_x\varphi\right]\text{d}x\text{d}t + \int_{-1}^{1}|\rho_0(x)-c|\, \varphi(0,x)\text{d}x\nonumber\\
	&+ \int_{0}^{+\infty}\left[f(\rho(t,1-))-f(c)\right]\varphi(t,1)\text{d}t + \int_{0}^{+\infty}\left[f(\rho(t,-1+))-f(c)\right]\varphi(t,-1)\text{d}t\\  &+
	2\int_{0}^{+\infty}f(c)\, \varphi(t,\xi(t))\text{d}t 
	 \geq 0, \nonumber 
	\end{align}
	where, recalling \eqref{eq:def-of-F},  
	\begin{align}\label{def:cal-F}
	    \mathcal{F}(t,x,\rho,c) \dot = & \sign(\rho - c)\left[F(t,x,\rho)-F(t,x,c)\right]\\
	    = & \sign(\rho - c) \sign(x-\xi(t))\left[f(\rho)-f(c)\right]. \nonumber
	\end{align}
\end{definition}
The first line in \eqref{eq:weak_solution} originates from the Kru\v{z}kov definition of entropy weak solution as would be in the case of a Cauchy problem, \cite{Kru_kov_1970}. The two terms in the second line come from the boundary condition introduced by Bardos et al. in \cite{c_bardos_1979} whereas the term in the last line accounts for the traces in the solution along the discontinuity in the flux. 

\subsection{The Riemann problem at turning curve}\label{sec:riemann} 
The solution of the Riemann problem serves as building block for the selection of a numerical scheme appropriate for the equation under consideration. Since our equation contains a jump at the flux interface, it is necessary to construct a Riemann solver at the interface. In this section, we present the Riemann solver for \eqref{eq:main} and detail the types of admissible elementary waves that are present in solution at $\xi(t)$.

Let $\xi(t) = \alpha t$ with $\alpha\in\eR$, and consider the Riemann problem
	\begin{equation}\label{eq:CL_alpha}
		\pt\rho + \px \left(\sign(x-\alpha t)f(\rho)\right) = 0,\qquad x\in \eR, \quad t>0
	\end{equation}
\begin{equation}
\label{eq:riemann_data}
\rho_0(x) = \begin{cases}
 \rho_L, &\text{if } x< 0, \\
 \rho_R, &\text{if } x> 0
\end{cases}
\end{equation}
with $\rho_L$, $\rho_R\in [0,1]$. We look for a self-similar solution of the problem \eqref{eq:CL_alpha}, \eqref{eq:riemann_data}
according to the Definition~\ref{def:Kruzhov-solutions}, that we will denote by $\mathcal{R}^{\alpha}(\rL, \rR)(t,x)$.

Also, we will denote by $\mathcal{R}^{\pm}(\rL, \rR)(t,x)$ the classical Riemann solver for 
		\begin{equation*}
		\pt\rho \pm  \px f(\rho) = 0,\qquad x\in \eR, \quad t>0
	\end{equation*}
with initial data \eqref{eq:riemann_data} at $t=0$.	


%

Let's define 
\begin{equation*}
    v(\rho) ~\dot =
    \begin{cases} ~\frac{f(\rho)}\rho\,, & \rho\in (0,R]\\
    \lim_{\rho\to 0+} \frac{f(\rho)}\rho  & \rho=0\,.
    \end{cases}
\end{equation*}
After the assumption \eqref{eq:hyp-on-f} on $f$, the function $v(\rho)$ is non-increasing and $v(0)$ is well defined and finite. 

In practice, we can distinguish between two cases:  either $\alpha \in [-v(\rho_L), v(\rho_R)]$ or not. 


\begin{itemize}
\item {{\fbox{$-v(\rho_L)\le \alpha\le v(\rho_R)$
}} }
In this case, for any couple $(\rL, \rR)\in [0,1]^2$, the function
\begin{equation}
\mathcal{R}^{\alpha}(\rL, \rR)(t,x) = 
\begin{cases}
\mathcal{R}^{-}(\rL, 0)(t,x), &\quad\text{ for } x<0,\\
\mathcal{R}^{+}(0, \rR)(t,x), &\quad\text{ for } x>0.\\
\end{cases}
\end{equation}
is a solution to \eqref{eq:CL_alpha}, \eqref{eq:riemann_data} in the sense of Definition~\ref{def:Kruzhov-solutions}.

More precisely, 
the solution takes the form 
\begin{equation}\label{eq:RS-with-vacuum}
\mathcal{R}^{\alpha}(\rL, \rR)(t,x) = 
\begin{cases}
\rL, &\quad\text{ for } x<-v(\rL)t,\\
0, &\quad\text{ for } x/t\in \left(-v(\rL), v(\rR) \right),\\
\rR,  &\quad\text{ for } x>v(\rR)t,\\
\end{cases}
\end{equation}
because the waves joining $\rL$ to $0$ and $0 $ to $\rR$, with flux $-f$ and $f$ respectively, are necessarily shock discontinuities. See Figure \ref{fig:riemann-at-xi}.

\begin{figure}[ht]
    \centering
\begin{tikzpicture}[x=0.7pt,y=0.7pt,yscale=-0.9,xscale=0.9]

\draw    (26.83,149) -- (302.32,149) ;
\draw [shift={(304.32,149)}, rotate = 180] [color={rgb, 255:red, 0; green, 0; blue, 0 }  ][line width=0.75]    (10.93,-3.29) .. controls (6.95,-1.4) and (3.31,-0.3) .. (0,0) .. controls (3.31,0.3) and (6.95,1.4) .. (10.93,3.29)   ;
\draw [line width=1.5]    (47.93,148) .. controls (107.91,30.5) and (199.23,-3.5) .. (282.28,148) ;
\draw [line width=1.5]    (47.93,148) .. controls (118.93,258.5) and (189,285.5) .. (282.28,148) ;
\draw [color={rgb, 255:red, 208; green, 2; blue, 27 }  ,draw opacity=1 ][line width=2.25]    (154,242) -- (215.37,64) ;
\draw [color={rgb, 255:red, 65; green, 117; blue, 5 }  ,draw opacity=1 ][line width=1.5]  [dash pattern={on 5.63pt off 4.5pt}]  (215.37,64) -- (47.93,148) ;
\draw [color={rgb, 255:red, 30; green, 6; blue, 220 }  ,draw opacity=1 ] [dash pattern={on 0.84pt off 2.51pt}]  (152.39,151) -- (153.96,238) ;
\draw [color={rgb, 255:red, 30; green, 6; blue, 220 }  ,draw opacity=1 ] [dash pattern={on 0.84pt off 2.51pt}]  (86.26,151) -- (86.26,197) ;
\draw [color={rgb, 255:red, 26; green, 7; blue, 236 }  ,draw opacity=1 ] [dash pattern={on 0.84pt off 2.51pt}]  (215.37,64) -- (215.37,150) ;
\draw    (368,144) -- (638.5,143.01) ;
\draw [shift={(640.5,143)}, rotate = 539.79] [color={rgb, 255:red, 0; green, 0; blue, 0 }  ][line width=0.75]    (10.93,-3.29) .. controls (6.95,-1.4) and (3.31,-0.3) .. (0,0) .. controls (3.31,0.3) and (6.95,1.4) .. (10.93,3.29)   ;
\draw    (374.5,220) -- (374.63,44) ;
\draw [shift={(374.63,42)}, rotate = 450.04] [color={rgb, 255:red, 0; green, 0; blue, 0 }  ][line width=0.75]    (10.93,-3.29) .. controls (6.95,-1.4) and (3.31,-0.3) .. (0,0) .. controls (3.31,0.3) and (6.95,1.4) .. (10.93,3.29)   ;
\draw [line width=1.5]    (376,143) .. controls (436.68,25.5) and (531.9,-8.5) .. (615.91,143) ;
\draw [line width=1.5]    (375,144) .. controls (446.83,254.5) and (521.55,280.5) .. (615.91,143) ;
\draw [color={rgb, 255:red, 208; green, 2; blue, 27 }  ,draw opacity=1 ][line width=2.25]    (518.5,233) -- (458.5,51.5) ;
\draw [color={rgb, 255:red, 65; green, 117; blue, 5 }  ,draw opacity=1 ][line width=1.5]  [dash pattern={on 5.63pt off 4.5pt}]  (518.5,233) -- (375,144) ;
\draw [color={rgb, 255:red, 30; green, 6; blue, 220 }  ,draw opacity=1 ] [dash pattern={on 0.84pt off 2.51pt}]  (516.5,144.5) -- (518.5,233) ;
\draw [color={rgb, 255:red, 30; green, 6; blue, 220 }  ,draw opacity=1 ] [dash pattern={on 0.84pt off 2.51pt}]  (563.5,74.5) -- (563.5,141.5) ;
\draw [color={rgb, 255:red, 26; green, 7; blue, 236 }  ,draw opacity=1 ] [dash pattern={on 0.84pt off 2.51pt}]  (458.5,51.5) -- (459,142.5) ;
\draw    (47.5,224) -- (47.63,48) ;
\draw [shift={(47.63,46)}, rotate = 450.04] [color={rgb, 255:red, 0; green, 0; blue, 0 }  ][line width=0.75]    (10.93,-3.29) .. controls (6.95,-1.4) and (3.31,-0.3) .. (0,0) .. controls (3.31,0.3) and (6.95,1.4) .. (10.93,3.29)   ;
\draw [color={rgb, 255:red, 11; green, 58; blue, 214 }  ,draw opacity=1 ][line width=0.75]    (458.5,51.5) -- (563.5,74.5) ;

\draw (16.37,66.4) node [anchor=north west][inner sep=0.75pt]  [font=\footnotesize]  {$f( \rho )$};
\draw (284.28,151.4) node [anchor=north west][inner sep=0.75pt]    {$\rho $};
\draw (84.52,129.9) node [anchor=north west][inner sep=0.75pt]  [font=\footnotesize,color={rgb, 255:red, 50; green, 18; blue, 224 }  ,opacity=1 ]  {$\rho _{L}$};
\draw (209.71,151.9) node [anchor=north west][inner sep=0.75pt]  [font=\footnotesize,color={rgb, 255:red, 33; green, 13; blue, 239 }  ,opacity=1 ]  {$\rho _{R}$};
\draw (146.65,133.4) node [anchor=north west][inner sep=0.75pt]  [font=\footnotesize,color={rgb, 255:red, 33; green, 13; blue, 239 }  ,opacity=1 ]  {$\rho _{M}$};
\draw (169.58,211.76) node [anchor=north west][inner sep=0.75pt]  [font=\footnotesize,color={rgb, 255:red, 33; green, 13; blue, 239 }  ,opacity=1 ,rotate=-288.5,xslant=-0.08]  {$\textcolor[rgb]{0.82,0.01,0.11}{\alpha  >v}\textcolor[rgb]{0.82,0.01,0.11}{(}\textcolor[rgb]{0.82,0.01,0.11}{\rho }\textcolor[rgb]{0.82,0.01,0.11}{_{R}}\textcolor[rgb]{0.82,0.01,0.11}{)}$};
\draw (113.83,96.27) node [anchor=north west][inner sep=0.75pt]  [font=\footnotesize,color={rgb, 255:red, 33; green, 13; blue, 239 }  ,opacity=1 ,rotate=-330.58]  {$\textcolor[rgb]{0.25,0.46,0.02}{v}\textcolor[rgb]{0.25,0.46,0.02}{(}\textcolor[rgb]{0.25,0.46,0.02}{\rho }\textcolor[rgb]{0.25,0.46,0.02}{_{R}}\textcolor[rgb]{0.25,0.46,0.02}{)}$};
\draw (343.1,66.4) node [anchor=north west][inner sep=0.75pt]  [font=\footnotesize]  {$f( \rho )$};
\draw (622, 147.2) node [anchor=north west][inner sep=0.75pt]    {$\rho $};
\draw (556.94,141.4) node [anchor=north west][inner sep=0.75pt]  [font=\footnotesize,color={rgb, 255:red, 50; green, 18; blue, 224 }  ,opacity=1 ]  {$\rho _{R}$};
\draw (439.7,142.4) node [anchor=north west][inner sep=0.75pt]  [font=\footnotesize,color={rgb, 255:red, 33; green, 13; blue, 239 }  ,opacity=1 ]  {$\rho _{M}$};
\draw (509.27,127.4) node [anchor=north west][inner sep=0.75pt]  [font=\footnotesize,color={rgb, 255:red, 33; green, 13; blue, 239 }  ,opacity=1 ]  {$\rho _{L}$};
\draw (483.06,145.97) node [anchor=north west][inner sep=0.75pt]  [font=\footnotesize,color={rgb, 255:red, 33; green, 13; blue, 239 }  ,opacity=1 ,rotate=-71.57,xslant=-0.08]  {$\textcolor[rgb]{0.82,0.01,0.11}{\alpha < -v}\textcolor[rgb]{0.82,0.01,0.11}{(}\textcolor[rgb]{0.82,0.01,0.11}{\rho }\textcolor[rgb]{0.82,0.01,0.11}{_{L}}\textcolor[rgb]{0.82,0.01,0.11}{)}$};
\draw (438.9,189.08) node [anchor=north west][inner sep=0.75pt]  [font=\footnotesize,color={rgb, 255:red, 33; green, 13; blue, 239 }  ,opacity=1 ,rotate=-31.86]  {$\textcolor[rgb]{0.25,0.46,0.02}{-v}\textcolor[rgb]{0.25,0.46,0.02}{(}\textcolor[rgb]{0.25,0.46,0.02}{\rho }\textcolor[rgb]{0.25,0.46,0.02}{_{L}}\textcolor[rgb]{0.25,0.46,0.02}{)}$};
\end{tikzpicture}

    \caption{\small{
    The Riemann solvers at turning curve for the initial data \eqref{eq:riemann_data}, with $\alpha>v(\rR)$ (left) and $\alpha<-v(\rL)$ (right).  %
     }
    }
    \label{fig:riemann-at-xi}
\end{figure}
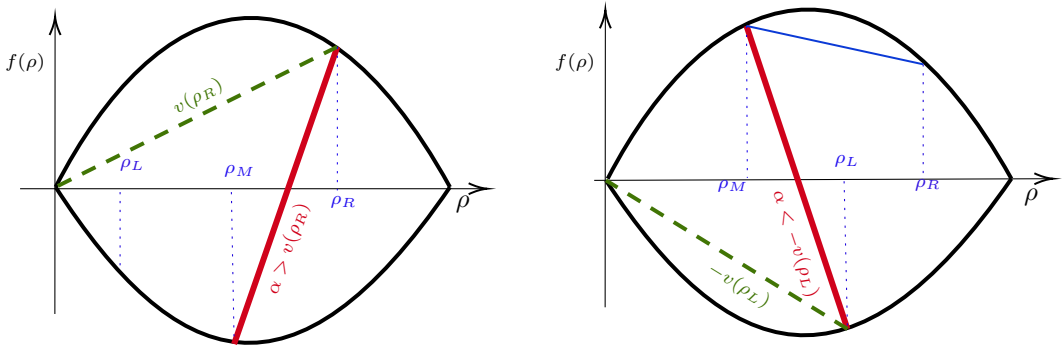
\item{\fbox{$\alpha\ge v(\rho_R)$}}
For this case 
we have
\begin{equation}\label{RS-without-vacuum}
\mathcal{R}^{\alpha}(\rL, \rR)(t,x) = 
\begin{cases}
\mathcal{R}^{-}(\rL, \rM) (t,x), &\quad\text{ for } x<\alpha t,\\
\rR,  &\quad\text{ for } x >\alpha t,\\
\end{cases}
\end{equation}
where $\rM$ is the only density value in $[0, \rR)$ such that the jump condition across $\xi(t)=\alpha t$,
\begin{equation}\label{eq:rhom}
    \alpha = \frac{f(\rR) + f(\rM)}{\rR-\rM}.
\end{equation}
is satisfied. 

\item{\fbox{$\alpha < -v(\rho_L)$}}  
On the other hand, if $\alpha <-v(\rL),$ the Riemann solver is similarly written
\begin{equation}\label{eq:RS-without-vacuum-neg-alpha}
\mathcal{R}^{\alpha}(\rL, \rR)(t,x) = 
\begin{cases}
\rL,  &\quad\text{ for } x<\alpha t,\\
\mathcal{R}^{+}(\rM, \rR) (t,x), &\quad\text{ for } x> \alpha t,
\end{cases}
\end{equation}
where the intermediate state value $\rM\in [0,\rL)$ satisfies
\begin{equation}\label{eq:rhom-neg-alpha}
    \alpha = \frac{f(\rM) + f(\rL)}{\rM-\rL}.
\end{equation}
\end{itemize}

See Figure \ref{fig:riemann-at-xi} above. The following lemma 
shows that the intermediate state value is unique. 

\begin{lemma}\label{lem:2.1}
	Let $f$ satisfy \eqref{eq:hyp-on-f}. 
	Then for any given  $\rR\in (0,1)$ and 
	$\alpha> v(\rR)$ there exists a unique $\rM\in (0, \rR)$ which satisfies \eqref{eq:rhom}.  
\end{lemma} 

\begin{proof} Define the map
$$\Phi(\rho) = \dfrac{f(\rho_R) + f(\rho)}{\rho_R-\rho}\,,\qquad \rho\in (0, \rR)\,.$$
  It is immediate to see that
  \begin{equation}
    \lim_{\rho\to 0+}\Phi(\rho) = f(\rho_R)/\rho_R = v(\rho_R), \qquad  \lim_{\rho\to \rho_R-}\Phi(\rho) = + \infty,
  \end{equation}
  and, by convexity of $f$, we have that
  \begin{equation}
    \Phi'(\rho) = \frac{f'(\rho) (\rho_R-\rho) + f(\rho_R) + f(\rho)}{(\rho_R-\rho)^2} \geq\frac{2f(\rho_R)}{(\rho_R-\rho)^2}>0.
  \end{equation}
  Therefore $\Phi$ is strictly increasing, and there exists a unique solution to \eqref{eq:rhom}. This completes the proof of the lemma.
\end{proof}

To ensure that the Riemann solver in \eqref{RS-without-vacuum} is well defined, a necessary condition is that for any couple $(\rL,\rR)$ and $\alpha > f(\rR)/\rR$ the speed of waves in $\mathcal{R}^{-}(\rL, \rM) $ is lower than $\alpha$. This condition is easily verified, thanks to the convexity of $-f$. 
\begin{proposition} The following bounds hold for all $\alpha\in \eR$,
\begin{equation}\label{eq:bound-on-RSalpha}
0\le \mathcal{R}^{\alpha}(\rL, \rR)(x/t) \le \max\{\rho_L,\rho_R\}\,.
\end{equation}
\end{proposition}

\begin{remark}
Since the flux switches sign, the solution to the Riemann problem 
satisfies the equation with convex flux $-f(\rho)$ on the domain $\Omega_L = \{(x,t): x < \alpha t\}$, and  with concave flux $+f(\rho)$
on $\Omega_R = \{(x,t):x > \alpha t\}$.

Suppose the pair $(\gamma_l\rho(t),\gamma_r\rho(t))$ are the left and right traces at $\xi(t)$ at each $t>0$, then definition \ref{def:Kruzhov-solutions} yields the Rankine-Hugoniot condition 
\begin{equation}
\label{eq:Rankine-hugoniot-main}
f(\gamma_r\rho(t)) + f(\gamma_l\rho(t)) = \xi'(t) (\gamma_r\rho(t) - \gamma_l\rho(t)),
\end{equation}
at $\xi(t)$. 

As a consequence of Definition \ref{def:Kruzhov-solutions}, if the solution is discontinuous at $x=\xi(t)$ the jump is undercompressive and the characteristic lines impinge on $\xi(t)$ from the side where the density is higher. More precisely,
if $\rho$ is an entropy solution and $(\gamma_l\rho(t),\gamma_r\rho(t))$ are its left and right traces at $\xi(t)$, at each $t>0$ we have 
\begin{equation}
    \min\{0, -\xi'(t)-f'(\gamma_l\rho(t))\}\max\{0, -\xi'(t)+f'(\gamma_r\rho(t))\}=0.
\end{equation}
This necessary condition has been pointed out in \cite[Prop. 5.1]{karlsen_risebro_towers} in a more general framework and in \cite[Prop. 2.4]{El-Khatib2013Weak} for the Hughes model. The numerical examples used in Section \ref{sec:NumericalExamples} well illustrates this effect.
\end{remark}

\subsection{The germ of admissible solutions}
The time-dependent discontinuity of $f$ in $\rho$ along the Lipschitz continuous curve $\xi$ makes it more appropriate to consider entropy solutions given in Definition \eqref{def:Kruzhov-solutions} within the admissibility germ framework. In this framework, the inequality \eqref{eq:weak_solution} reduces to a pair of traces at $\xi$ if the condition \eqref{eq:Rankine-hugoniot-main} is satisfied. It is important to note that since $\pm f$ are assumed to be continuous over $[0,R]^2$ and the measure of the set $\{ s\in [0, 1] \text{ s.t. } f'(s)=0 \}$ is zero, the traces $\gamma^{l,r}\rho(t)$ are actually one-sided strong traces, as proven in \cite[Theorem 2.1]{Andreianov2011}. This notion of admissibility germs relies on the idea that Kruzhkov's entropy is satisfied away from $\xi$, while the "non-standard" Rankine-Hugoniot condition is adapted at $\xi$. Thanks to the germs, it should be possible to define the notion of an entropy process solution. We proceed to define the germ below.

 \begin{definition}
 \label{def:connection}
 Let $\xi(t) = \alpha t $, with $\alpha\in \eR$ fixed. The germ of admissible solutions $\mG_{\alpha}$ for the conservation law \eqref{eq:main} is
 \begin{equation}
  \mG_{\alpha} :=  \left\{ (0, \rho_{\alpha}) \in [0, 1]^2 : f(\rho_{\alpha}) = \alpha\rho_{\alpha}\right\} .
 \end{equation}
 \end{definition}
 \begin{definition}
 A germ $\mG$ is $L^1$-dissipative ($\LD$ for short), if for any $(p_l, p_r)$, $(q_l, q_r)\in \mG(\xi)$ the following dissipative property holds
\begin{equation}
    \label{ineq:ch3-dissipative}
    \mF(t,x,p_l, q_l) - \mF(t, x, p_r, q_r) \geq \alpha(\eta(p_l, q_l) -        \eta(p_r, q_r)),
\end{equation}
where $ \mF(t, x, \rho, c)$ is defined in Definition \eqref{def:Kruzhov-solutions} and $\eta(\rho, c) := |\rho - c|$. 
\end{definition}
\begin{proposition}
 The unique maximal $\LD$ extension of $\mG_{\alpha}$ is the subset of $[0, 1]^2$ defined by
\begin{equation}
    \widetilde\mG_{\alpha} := \left\{ (p_l, p_r) \in [0, 1]^2 : f(p_r) + f(p_l) = \alpha(p_r-p_l)\right\}.
\end{equation}
 \end{proposition}
 \begin{proof}
 The fact that $\widetilde\mG_{\alpha}$ contains $\mG_{\alpha}$ is obvious as 
 \[
 f(\rho_{\alpha})+ f(0)= \alpha\rho_{\alpha} -0.
 \]
 To show that $\widetilde\mG_{\alpha}$ is $\LD$ we distinguish three cases : $\alpha=0$, $\alpha>0$, and $\alpha<0$. 
\begin{itemize}
\item[If $\alpha=0$] the germ only contains four elements $(0,0)$,  $(0,1)$, $(1,0)$, $(1,1)$ and we can check that \eqref{ineq:ch3-dissipative} holds as an equality by direct computations.
\item[If $\alpha>0$] we consider two sub-cases
\begin{itemize}
    \item[I.] If $q_l=q_r=0$ we have 
    \begin{align*}
        \mF(t, x, p_l, 0) - \mF(t, x, p_r, 0) &= (-f(p_l) + f(0)) - (f(p_r) - f(0)) \\
         &=-(f(p_r) + f(p_l)) = \alpha (p_l-q_r)\\
         & = \alpha(\eta(p_l, 0) - \eta(p_r, 0))\,,
    \end{align*}
    for all $p_l,p_r\in [0,1]^2$. The case in which $q_l=q_r=1$ is similar. 
    \item[II.] If $q_l<q_r$  (the reverse inequality being impossible due to the sign of $\alpha$ and the geometry of the problem) we can observe that 
    \[
    \sign(p_l-q_l) = \sign(p_r-q_r). 
    \]
    Therefore
    \begin{align*}
        \mF(t, x, p_l, q_l) - \mF(t, x, p_r, q_r) &= \sign(p_l-q_l)\left(-f(p_l) + f(q_l) - f(p_r) + f(q_r)\right) \\
         &=\sign(p_l-q_l)\alpha \left( q_r-q_l -p_r+p_l\right)\\
         & = \alpha(\eta(p_l, q_l) - \eta(p_r, q_r))\,.
    \end{align*}
    \item[If $\alpha<0$] the analysis is exactly the same as in the previous case.      
\end{itemize}
\end{itemize}
It is also clear that $\widetilde \mG_{\alpha}$ is maximal: given any $\LD$ germ $\mG$ one can look for its possible $\LD$ extensions, the largest of which is its dual $\mG^*$, i.e. the set containing all couples $(q_l, q_r) \in [0, R]^2$ which satisfy the Rankine-Hugoniot conditions and \eqref{ineq:ch3-dissipative} for any $(p_l, p_r)\in\mG$. Of course, $\widetilde \mG_{\alpha}$ and $\widetilde \mG_{\alpha}^*$ coincide.
 \end{proof}

Additionally, given any piecewise constant initial condition $\rho_l\chi_{x<\xi(t)} + \rho_r\chi_{x>\xi(t)}$, with $\rho_l$ and $\rho_r\in [0,1]$, then the left and right side traces of the solution of the associate Riemann problem correspond to an element of $\widetilde\mG_{\alpha}$. More formally

\begin{lemma}
    \label{prop:property-of-germs}
    For any $\xi=\alpha t$, the $ \LD$ germ $\widetilde\mG_{\alpha}$ is complete.
\end{lemma}
\begin{proof}
This lemma follows immediately from the study of the Riemann problem at the interface in \ref{sec:riemann}.
\end{proof}
We now define $\mG$-entropy and $\mG$-entropy process solutions for our problem below.
\begin{definition}
\label{def:G-solutions} Assume that $f$, $\xi$ satisfy \eqref{eq:hyp-on-f}, \eqref{eq:hyp-on-xi} respectively. We say that map $(t,x)\mapsto\rho(t,x)$ is a $\mG$-entropy solution to the initial-boundary value problem \eqref{eq:main}, if $\rho$ in $\Lp{\infty}\left([0,+\infty[\times\Omega; ~[0,1]\right)$, is a weak solution of this problem and for any $(p_l,p_r)\in \mG$, and any test function $\varphi \in \mathbf{C}^{\infty}_{\mathbf{c}}\left(\eR^2; ~[0,+\infty[\right)$ we have 
	\begin{align}
	\label{eq:G_solution}
	&\int_{0}^{+\infty}\!\!\int_{-1}^{1}\left[|\rho - c(x)|\partial_t\varphi + \mathcal{F}(t,x,\rho,c(x))\partial_x\varphi\right]\text{d}x\text{d}t + \int_{-1}^{1}|\rho_0(x)-c(x)|\varphi(0,x)\text{d}x\nonumber\\
	&+ \int_{0}^{+\infty}\left[f(\rho(t,1-))-f(c(x))\right]\varphi(t,1)\text{d}t ~+ \\  &+\int_{0}^{+\infty}\left[f(\rho(t,-1+))-f(c(x))\right]\varphi(t,-1)\text{d}t \geq 0, \nonumber 
	\end{align}
	where $c(x) = p_l\chi_{x<\xi(t)} (x)+ p_r\chi_{x>\xi(t)}(x)$ and $\mathcal{F}(t,x,\rho,c)$ is defined at \eqref{def:cal-F}.
\end{definition}
\begin{remark} Following \cite[Remark 3.12 and Th. 3.18]{Andreianov2011} we recall that
\begin{itemize}
    \item $\rho \in\Lp{\infty}\left([0,+\infty[\times\Omega; ~[0,1]\right)$ is a $\mG_\alpha$-entropy solution if and only if it is an entropy solution in the sense of Kruzhkov on $[0,+\infty[\times [-1,\xi^-)$ and $[0,+\infty[\times (\xi^+, 1]$,  and for almost every $t>0$ the couple given by its traces at $x=\xi(t)$, $(\gamma^{l}\rho(t),\gamma^{r}\rho(t))$, is in $\widetilde \mG_\alpha = \mG_\alpha^*$. From the definition of $\widetilde \mG_\alpha$ it is clear that this last point is equivalent to say that $\rho$ is a weak solution on $[0,+\infty[\times\Omega$. Therefore, Definition \ref{def:G-solutions} is equivalent to Definition 1 in \cite{Amadori2012}.
    \item In our setting $\mG_\alpha$- and $\widetilde \mG_\alpha$-entropy solutions coincide. 
\end{itemize}
\end{remark}
\begin{definition}
\label{def:process-solutions} 
Assume that $f$, $\xi$ satisfy \eqref{eq:hyp-on-f}, \eqref{eq:hyp-on-xi} respectively. We say that a function $\mu \in \Lp{\infty}\left([0,+\infty[\times\Omega\times (0,1); ~[0,1]\right)$ is a $\widetilde \mG_\alpha$-entropy process solution to the initial-boundary value problem \eqref{eq:main}, if the following conditions hold : 
\begin{enumerate}
    \item For any test function $\varphi \in \mathbf{C}^{\infty}_{\mathbf{c}}\left(\eR^+\times(-1,1); ~[0,+\infty[\right)$ we have 
    \begin{align}
	\label{eq:weakproc-solution}
	  \int_0^1\int_{0}^{+\infty}\!\!\int_{-1}^{1}\left[\mu(t,x,a) \partial_t\varphi + F(t,x,\mu(t,x,a))\partial_x\varphi\right]\text{d}x\text{d}t \text{d}a+ \int_{-1}^{1}\rho_0(x)\varphi(0,x)\text{d}x= 0, \nonumber 
	\end{align}
	\item  The inequality 
	\begin{align}
	&\int_0^1\int_{0}^{+\infty}\!\!\int_{-1}^{1}\left[|\mu(t,x,a) - c(x)|\partial_t\varphi + \mathcal{F}(t,x,\mu(t,x,a),c(x))\partial_x\varphi\right]\text{d}x\text{d}t\text{d}a \\ &+ \int_{-1}^{1}|\rho_0(x)-c(x)|\varphi(0,x)\text{d}x\geq 0, \nonumber 
 	\end{align}
	where  $\mathcal{F}(t,x,\rho,c) = \sign(\rho - c)\left[F(t,x,\rho)-F(t,x,c)\right]$, is satisfied 
	\begin{itemize}
	    \item for any test function $\varphi \in \mathbf{C}^{\infty}_{\mathbf{c}}\left(\eR^+\times(-1,1); ~[0,+\infty[\right)$ and for any $c(x) = p_l\chi_{x<\xi(t)} (x)+ p_r\chi_{x>\xi(t)}(x)$ such that $(p_l,p_r)\in \widetilde \mG_\alpha$ ;
	    \item for any test function $\varphi \in \mathbf{C}^{\infty}_{\mathbf{c}}\left(\eR^+\times(-1,1); ~[0,+\infty[\right)$ such that $\varphi = 0$ on $x=\xi(t)$ and any constant $c(x)=c\in \mathbb{R}$.
	\end{itemize}
\end{enumerate}
\end{definition}
\begin{remark}
We could define $\mG_\alpha$-entropy process solutions, but since they do not coincide in general with $\widetilde \mG_\alpha$-entropy process solutions, we do not use them here. The fact that $\widetilde \mG_\alpha$ is maximal $\LD$ is a necessary hypothesis in the following theorem, see \cite[Th: 3.28]{Andreianov2011}.
\end{remark}
\begin{theorem}
Let $\mG$ be a maximal $\LD$ germ and let $\rho_0$ be an initial condition for which the initial-boundary value problem \eqref{eq:main} admits a $\mG$-entropy solution, $\rho$. Then there exists a unique $\mG$-entropy process solution $\mu$ associated to the initial condition $\rho_0$ and $\mu(\alpha) = \rho$ for almost every $\alpha \in (0,1)$.
\end{theorem}
Thanks to the existence result proved in \cite{karlsen_risebro_towers} we deduce that our problem admits a unique $\widetilde\mG_\alpha$-entropy process solution. We will construct this in the next section as the limit of finite volume approximations. 

\section{The Finite Volume Scheme}
\label{sec:FiniteVolumeScheme}
In this section, we present a finite volume scheme with a moving mesh adapted near $ \xi$ to approximate the weak solution of equation \eqref{eq:main}. 
Our strategy to define the moving mesh consists in discretizing the space-time domain into a reference mesh, but at each time only cells closer to the flux interface $\xi$ are adapted such that it is included in the computational mesh. The main feature of the computational (or moving) mesh is that the two cells adjacent to $\xi$ have variable
lengths, whereas the cells on the reference mesh have a fixed cell length, $\DX$. Modified marching formulas, which incorporate a modified numerical flux function based on the non-classical Riemann solver $(\mathcal{R}^{\alpha})$, are applied to the two cells adjacent to the turning curve. The standard Godunov marching formulas are used for cells away from $\xi$, where the reference mesh is established.

This is a slight variation from the mesh adaptation technique presented in \cite{ZHONG1996192} and used in \cite{dellemonache2014} for a traffic flow model with a moving constraint. This technique involves replacing the cell interface closest to $\xi$ and shifting the adjacent interfaces, followed by a Lagrange interpolation formula to calculate a new temporal density average for two adjacent cells. Although their scheme is conservative, the recomputation of the new density averages makes it difficult to analyze the well-balanced property of the resulting scheme. The scheme proposed here does not suffer from this severity. Another variation of the discretization procedure also adapted for the traffic flow model with moving constraints can be found in \cite{sylla2021lwr}. In the next subsections, we detail our space discretization procedure and present the numerical scheme. 

\subsection{Discretization}
Fix $N\in\eN$ and define $\DX =2^{-N}$ as the space step. Let $x_{j+1/2} : = (j+ 1/2) \DX$ be a point of the reference mesh and $C_j:= \left[x_{j-1/2}, x_{j+1/2}\right)$ be a computational cell.

Since we are approximating the space interval $\Omega = ]-1,1[$, the indexes $j\in \eN$ of the cells are considered 
in the range $j=-\frac1\DX,\ldots,\frac 1 \DX$\,.
The leftmost and rightmost cells are 
$$
C_{-\frac 1\DX}:= \left[x_{-\frac 1\DX-\frac12}, x_{-\frac 1 \DX+\frac12}\right)\,,\qquad  C_{\frac 1\DX}:= \left[x_{\frac 1\DX-\frac12}, x_{\frac 1\DX+\frac12}\right)\,.
$$

The computational cells here above 
contain the boundary points $x = -1$ and $x=1$ respectively, and
\begin{equation}\label{eq:union-Cj-equals-I}
    \bigcup_{j=- 1/\DX}^{1/\DX} C_j = \left[-\frac \DX 2 -1,1+\frac \DX 2 \right)\,.
\end{equation} 
\noindent 
Let $\DT > 0$ denote the time step and set 
$t^{n} = n\DT$ for $n\in\eN$. 
We define the slope discretization of $\xi(t)$ as follows, 
\begin{equation}\label{eq:def-alpha_n}
\alpha_{n,\DT} =\frac{1}{\DT}\int_{t^n}^{t^{n+1}}\dot{\xi}(t)\, dt, \qquad 
n \in\eN  
\end{equation}

and by $\xi^\DT(t)$ the polygonal function that interpolates linearly the points $(t^n,\xi(t^n))$, that is
\begin{equation}\label{def:xi-DT}
  \xi^\DT(t^n)  = \xi(t^n)=:\xi^n \,,\qquad \frac d{dt}\xi^\DT (t) =  \alpha_{n,\DT} \qquad 
  t\in (t^n,t^{n+1})\,.
\end{equation}

The next definition provides the basic notations and settings for the \textbf{moving mesh}.

\begin{definition}\label{def:the-mesh} 
Let the space step $\DX=2^{-N}>0$ and time step $\DT>0$ be fixed. Recalling \eqref{eq:hyp-on-xi}, choose $N\ge N_0$ with $N_0\in\eN$ large enough so that 
$\|\xi\|_\infty < 1 - \DX/2$ for every $N\ge N_0$\,.
For every $n\in\eN$,  let $m(n)\in 
\{-\frac 1\DX +1,\ldots,\frac 1\DX-1\}
$ be the single value such that 
\begin{equation}\label{xi-n-where-it-is}
    \xi^n\in\left[x_{m(n)-1/2}, ~x_{m(n)+1/2}\right)\,.
\end{equation} 
At a time $t^n$ the \ul{moving mesh} 
is defined by the points 
\begin{equation}
\label{eq:moving-mesh-n}
\xx^n_{j-1/2} := \begin{cases}
x_{j-1/2}, &\ j\leq m(n)-1 \mbox{ or } j\geq m(n)+2\,,\\
\xi^n, &\   
j = m(n),\ m(n)+1\,.
\end{cases}
\end{equation}
The computational cells in the moving mesh at each time $t^n$ are defined by 
\begin{equation}\label{def:new-cells}
    j\not = m(n)\,,\ |j|\le \frac 1\DX\,,\qquad \CCC^n_j := 
    [\xx^n_{j-1/2},\xx^n_{j+1/2}) =
    \begin{cases} 
        C_{j}, & 
        |j-m(n)|\ge 2\\
        D^n_L, & 
        j = m(n) -1 ,\\
        D^n_R, & 
        j = m(n) +1,\\
    \end{cases}
\end{equation}
where 
\begin{equation}\label{def:DnLR}
    D^n_L := \CCC^n_{m(n)-1} =  \left[x_{m(n)-3/2}, \, \xi^n\right)\,, \qquad 
D_R^n := \CCC^n_{m(n)+1} = \left[\xi^n,\, x_{m(n)+3/2}\right)\,.
\end{equation}
The length of the computational cells $\CCC^n_j$ in the moving mesh is given by 
\begin{equation}\label{def:cell-size}
    j\not = m(n)\,,\qquad  \DX^n_j := |\CCC^n_j| =
    \begin{cases} 
        \DX , & 
        |j-m(n)|\ge 2\\
        \Delta^n_L, & 
        j = m(n) -1 ,\\
        \Delta^n_ R, & 
        j = m(n) +1,\\
    \end{cases}
\end{equation}
where, after \eqref{xi-n-where-it-is} and \eqref{def:DnLR}, the quantities
$\Delta^n_L :=\left| D^n_L\right|$ \,, 
$\Delta^n_ R := \left| D^n_R\right|$
satisfy
\begin{align}
\Delta^n_L + \Delta^n_R  = 3\DX\,,\qquad  \DX \le  \Delta^n_L,~~ \Delta^n_R    \le 2\DX\,.
     \end{align}
\end{definition}
We stress that $\xi^n$ is a point of the grid at time $t^n$ and that it corresponds to the lumped indexes $m(n)\pm1$. Therefore, the two nearby cells around the value $\xi^n$ correspond to three cells in the reference mesh.
Comparing with \eqref{eq:union-Cj-equals-I}, we notice that
\begin{equation}\label{eq:union-Cnj-equals-I}
     \bigcup_{j=- 1/\DX,\ j\not = m(n)}^{1/\DX} \CCC^n_j = \left[-\frac \DX 2 -1,1+\frac \DX 2 \right)\,.
\end{equation}

\subsection{Time evolution of the grid} Here we consider 
the time variation of the grid from $t^n$ to $t^{n+1}$.
By assuming the CFL condition
\begin{equation}\label{eq:CFL-cond}
\DT \max\left\{\norm{f'}_{\infty}, \norm{\xi'}_{\infty} \right\} \leq \frac{\DX}{2}\,,
\end{equation}
and recalling \eqref{def:xi-DT}, we find that
$$\left| \xi^{n+1} - \xi^n\right|\leq \DT \norm{\xi'}_{\infty}  \le   \frac\DX 2$$
and therefore the map $\xi^\DT(t)$ crosses at most one boundary cell of the reference mesh. That is, 
at $t=t^{n+1}$ one has 
$$
m(n+1) \in \{m(n), m(n)\pm 1\}\,.
$$

Here we analyze the variation of the mesh, given in Definition~\ref{def:the-mesh}, from time $t^n$ to $t^{n+1}$ 
in the various cases.

\medskip 
\boxed{\textbf{Case A:}\ 
m(n)=m(n+1)} \quad 
In this case 
$
\xx^{n+1}_{j-1/2} = \xx^n_{j-1/2}$     for $j\not = m(n),\ m(n)+1$ 
and the point $\xi^n$ moves to $\xi^{n+1}$. In particular:
\begin{align*}
    t^{n+1}:\qquad &x_{m(n+1)-3/2}\,, \qquad 
    {\bord \xi^{n+1}}\,, \qquad x_{m(n+1)+3/2}\,;\\
    t^n:\qquad &x_{m(n)-3/2}\,, \qquad\quad {\bord \xi^n}\,, \qquad \ \ \, x_{m(n)+5/2}  \,.
\end{align*}
The values of $\Delta^n_L$,  $\Delta^{n+1}_L$ and of $\Delta^n_R$,  $\Delta^{n+1}_R$  satisfy the relations
\begin{equation}\label{eq:lenght-cellsLR-A}
    \Delta^{n+1}_L = \Delta^n_L + \alpha_{n,\DT}\DT\,,\qquad \Delta^{n+1}_R = \Delta^n_R - \alpha_{n,\DT}\DT\,.
\end{equation}

See Figure \ref{fig:moving-mesh-A}.
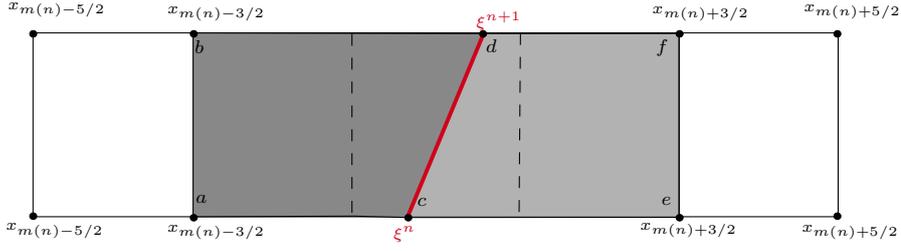
\begin{figure}
    \centering

\begin{tikzpicture}[x=0.75pt,y=0.75pt,yscale=-1,xscale=1]

\draw   (144.72,39.65) -- (546.64,39.65) -- (546.64,132.11) -- (144.72,132.11) -- cycle ;
\draw    (546.64,39.95) -- (546.64,132.42) ;
\draw  [fill={rgb, 255:red, 128; green, 128; blue, 128 }  ,fill opacity=0.93 ] (369.95,40.03) -- (331.26,132.5) -- (305.33,131.8) -- (224.63,132.11) -- (224.63,39.65) -- cycle ;
\draw  [fill={rgb, 255:red, 155; green, 155; blue, 155 }  ,fill opacity=0.77 ] (467.18,39.95) -- (467.18,132.42) -- (427.2,132.42) -- (331.26,132.5) -- (369.26,40.03) -- cycle ;
\draw  [dash pattern={on 4.5pt off 4.5pt}]  (303.95,40.11) -- (303.95,132.57) ;
\draw  [dash pattern={on 4.5pt off 4.5pt}]  (388.05,39.7) -- (387.36,132.16) ;
\draw [color={rgb, 255:red, 208; green, 2; blue, 27 }  ,draw opacity=1 ][line width=1.5]    (369.95,39.26) -- (331.26,133.27) ;
\draw    (467.28,39.95) -- (467.28,132.42) ;
\draw  [fill={rgb, 255:red, 0; green, 0; blue, 0 }  ,fill opacity=1 ] (223.25,132.59) .. controls (223.25,131.69) and (224.02,130.96) .. (224.98,130.96) .. controls (225.93,130.96) and (226.7,131.69) .. (226.7,132.59) .. controls (226.7,133.5) and (225.93,134.23) .. (224.98,134.23) .. controls (224.02,134.23) and (223.25,133.5) .. (223.25,132.59) -- cycle ;
\draw  [fill={rgb, 255:red, 0; green, 0; blue, 0 }  ,fill opacity=1 ] (223.25,40.13) .. controls (223.25,39.22) and (224.02,38.49) .. (224.98,38.49) .. controls (225.93,38.49) and (226.7,39.22) .. (226.7,40.13) .. controls (226.7,41.03) and (225.93,41.76) .. (224.98,41.76) .. controls (224.02,41.76) and (223.25,41.03) .. (223.25,40.13) -- cycle ;
\draw  [fill={rgb, 255:red, 0; green, 0; blue, 0 }  ,fill opacity=1 ] (143.1,40.13) .. controls (143.1,39.22) and (143.88,38.49) .. (144.83,38.49) .. controls (145.79,38.49) and (146.56,39.22) .. (146.56,40.13) .. controls (146.56,41.03) and (145.79,41.76) .. (144.83,41.76) .. controls (143.88,41.76) and (143.1,41.03) .. (143.1,40.13) -- cycle ;
\draw  [fill={rgb, 255:red, 0; green, 0; blue, 0 }  ,fill opacity=1 ] (143.1,131.82) .. controls (143.1,130.92) and (143.88,130.19) .. (144.83,130.19) .. controls (145.79,130.19) and (146.56,130.92) .. (146.56,131.82) .. controls (146.56,132.73) and (145.79,133.46) .. (144.83,133.46) .. controls (143.88,133.46) and (143.1,132.73) .. (143.1,131.82) -- cycle ;
\draw  [fill={rgb, 255:red, 0; green, 0; blue, 0 }  ,fill opacity=1 ] (330.34,132.59) .. controls (330.34,131.69) and (331.11,130.96) .. (332.07,130.96) .. controls (333.02,130.96) and (333.79,131.69) .. (333.79,132.59) .. controls (333.79,133.5) and (333.02,134.23) .. (332.07,134.23) .. controls (331.11,134.23) and (330.34,133.5) .. (330.34,132.59) -- cycle ;
\draw  [fill={rgb, 255:red, 0; green, 0; blue, 0 }  ,fill opacity=1 ] (465.76,132.59) .. controls (465.76,131.69) and (466.53,130.96) .. (467.48,130.96) .. controls (468.44,130.96) and (469.21,131.69) .. (469.21,132.59) .. controls (469.21,133.5) and (468.44,134.23) .. (467.48,134.23) .. controls (466.53,134.23) and (465.76,133.5) .. (465.76,132.59) -- cycle ;
\draw  [fill={rgb, 255:red, 0; green, 0; blue, 0 }  ,fill opacity=1 ] (465.76,40.13) .. controls (465.76,39.22) and (466.53,38.49) .. (467.48,38.49) .. controls (468.44,38.49) and (469.21,39.22) .. (469.21,40.13) .. controls (469.21,41.03) and (468.44,41.76) .. (467.48,41.76) .. controls (466.53,41.76) and (465.76,41.03) .. (465.76,40.13) -- cycle ;
\draw  [fill={rgb, 255:red, 0; green, 0; blue, 0 }  ,fill opacity=1 ] (544.52,40.13) .. controls (544.52,39.22) and (545.29,38.49) .. (546.25,38.49) .. controls (547.2,38.49) and (547.97,39.22) .. (547.97,40.13) .. controls (547.97,41.03) and (547.2,41.76) .. (546.25,41.76) .. controls (545.29,41.76) and (544.52,41.03) .. (544.52,40.13) -- cycle ;
\draw  [fill={rgb, 255:red, 0; green, 0; blue, 0 }  ,fill opacity=1 ] (544.52,131.82) .. controls (544.52,130.92) and (545.29,130.19) .. (546.25,130.19) .. controls (547.2,130.19) and (547.97,130.92) .. (547.97,131.82) .. controls (547.97,132.73) and (547.2,133.46) .. (546.25,133.46) .. controls (545.29,133.46) and (544.52,132.73) .. (544.52,131.82) -- cycle ;
\draw  [fill={rgb, 255:red, 0; green, 0; blue, 0 }  ,fill opacity=1 ] (367.65,40.13) .. controls (367.65,39.22) and (368.42,38.49) .. (369.38,38.49) .. controls (370.33,38.49) and (371.1,39.22) .. (371.1,40.13) .. controls (371.1,41.03) and (370.33,41.76) .. (369.38,41.76) .. controls (368.42,41.76) and (367.65,41.03) .. (367.65,40.13) -- cycle ;

\draw (224.66,119.98) node [anchor=north west][inner sep=0.75pt]  [font=\fontsize{0.71em}{0.85em}\selectfont]  {$a$};
\draw (224.2,42.02) node [anchor=north west][inner sep=0.75pt]  [font=\fontsize{0.71em}{0.85em}\selectfont]  {$b$};
\draw (335.11,121.54) node [anchor=north west][inner sep=0.75pt]  [font=\fontsize{0.71em}{0.85em}\selectfont]  {$c$};
\draw (369.52,41.64) node [anchor=north west][inner sep=0.75pt]  [font=\fontsize{0.71em}{0.85em}\selectfont]  {$d$};
\draw (457.1,120.59) node [anchor=north west][inner sep=0.75pt]  [font=\fontsize{0.71em}{0.85em}\selectfont]  {$e$};
\draw (454.39,41.79) node [anchor=north west][inner sep=0.75pt]  [font=\fontsize{0.71em}{0.85em}\selectfont]  {$f$};
\draw (210.53,134.85) node [anchor=north west][inner sep=0.75pt]  [font=\fontsize{0.59em}{0.71em}\selectfont]  {$x_{m( n) -3/2}$};
\draw (323.28,135.7) node [anchor=north west][inner sep=0.75pt]  [font=\fontsize{0.59em}{0.71em}\selectfont,color={rgb, 255:red, 208; green, 2; blue, 27 }  ,opacity=1 ]  {$\textcolor[rgb]{0.82,0.01,0.11}{\xi }\textcolor[rgb]{0.82,0.01,0.11}{^{n}}$};
\draw (364.6,27.83) node [anchor=north west][inner sep=0.75pt]  [font=\fontsize{0.59em}{0.71em}\selectfont,color={rgb, 255:red, 208; green, 2; blue, 27 }  ,opacity=1 ]  {$\textcolor[rgb]{0.82,0.01,0.11}{\xi }\textcolor[rgb]{0.82,0.01,0.11}{^{n+1}}$};
\draw (527.32,134.85) node [anchor=north west][inner sep=0.75pt]  [font=\fontsize{0.59em}{0.71em}\selectfont]  {$x_{m( n) +5/2}$};
\draw (129.92,133.97) node [anchor=north west][inner sep=0.75pt]  [font=\fontsize{0.59em}{0.71em}\selectfont]  {$x_{m( n) -5/2}$};
\draw (446.62,133.92) node [anchor=north west][inner sep=0.75pt]  [font=\fontsize{0.59em}{0.71em}\selectfont]  {$x_{m( n) +3/2}$};
\draw (130.92,22.97) node [anchor=north west][inner sep=0.75pt]  [font=\fontsize{0.59em}{0.71em}\selectfont]  {$x_{m( n) -5/2}$};
\draw (210.53,24.85) node [anchor=north west][inner sep=0.75pt]  [font=\fontsize{0.59em}{0.71em}\selectfont]  {$x_{m( n) -3/2}$};
\draw (452.62,24.92) node [anchor=north west][inner sep=0.75pt]  [font=\fontsize{0.59em}{0.71em}\selectfont]  {$x_{m( n) +3/2}$};
\draw (528.32,23.85) node [anchor=north west][inner sep=0.75pt]  [font=\fontsize{0.59em}{0.71em}\selectfont]  {$x_{m( n) +5/2}$};

\end{tikzpicture}
    \caption{Illustration of mesh adaptation for Case \textbf{A} }
    \label{fig:moving-mesh-A}
\end{figure}

\medskip
\boxed{\textbf{Case B:}\ m(n+1) = m(n)+1}\quad Here 
the line joining $(\xi^n, t^n)$ and $(\xi^{n+1}, t^{n+1})$ crosses the vertical line $x = x_{m(n)+1/2}$ of the reference mesh. See Figure \ref{fig:moving-mesh-B-def-of-rho}. 

To illustrate the mesh change from time $t^n$ to $t^{n+1}$, let us consider the points corresponding to the four indexes $j=m(n)-1\,,\ldots\,,m(n)+2$:
\begin{align*}
    t^{n+1}:\qquad &x_{m(n+1)-5/2}\,, \qquad x_{m(n+1)-3/2}\,,\quad\quad\quad
    {\bord \xi^{n+1}}\,, \qquad x_{m(n+1)+3/2}\,;\\
    t^n:\qquad &x_{m(n)-3/2}\,, \qquad\quad {\bord \xi^n}\,,\qquad\quad\quad\quad x_{m(n)+3/2}\,, \qquad x_{m(n)+5/2}  \,.
\end{align*}
Note that the grid point $x_{m(n)-1/2} = x_{m(n+1)-3/2}$ is present at time $t^{n+1}$ and not at time $t^n$, whereas the grid point $x_{m(n)+3/2} =x_{m(n+1)+1/2}$ is present at time $t^n$ but it is excluded from the mesh at $t^{n+1}$.

The following identities hold:
\begin{equation}
\label{caseB-id_DX}
\DX + \Delta^{n+1}_L= \Delta^n_L +\alpha_{n,\DT} \DT\,,\qquad 
\Delta^{n+1}_R= \Delta^n_R -\alpha_{n,\DT} \DT + \DX\,.    \end{equation}

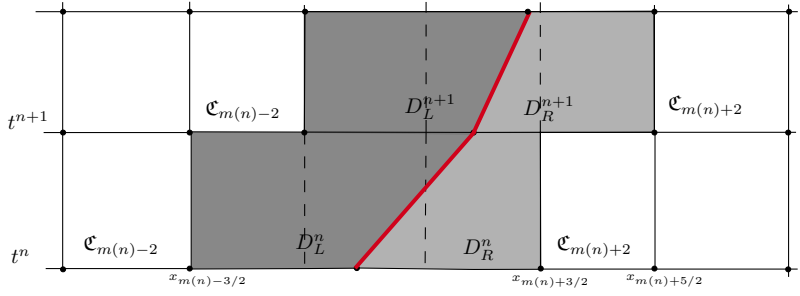
\begin{figure}[htbp]
    \centering

\begin{tikzpicture}[x=0.75pt,y=0.75pt,yscale=-0.8,xscale=0.8]

\draw  [fill={rgb, 255:red, 155; green, 155; blue, 155 }  ,fill opacity=0.77 ] (467.13,85.87) -- (467.4,162) -- (430.52,162.11) -- (355.18,162.18) -- (389.73,86.13) -- cycle ;
\draw  [fill={rgb, 255:red, 128; green, 128; blue, 128 }  ,fill opacity=0.93 ] (390.29,86.13) -- (353.62,163.18) -- (287.76,162.47) -- (249.03,162.18) -- (249.42,86.13) -- cycle ;
\draw    (467.81,162.02) -- (467.81,248.26) ;
\draw  [fill={rgb, 255:red, 128; green, 128; blue, 128 }  ,fill opacity=0.93 ] (355.52,162.09) -- (280.73,247.47) -- (251.08,247.69) -- (178.4,247.97) -- (178.4,161.73) -- cycle ;
\draw  [fill={rgb, 255:red, 155; green, 155; blue, 155 }  ,fill opacity=0.77 ] (396.25,162.02) -- (396.52,248.15) -- (354.64,248.26) -- (280.29,247.33) -- (355.52,162.09) -- cycle ;
\draw  [dash pattern={on 4.5pt off 4.5pt}]  (249.21,162.16) -- (249.21,248.4) ;
\draw  [dash pattern={on 4.5pt off 4.5pt}]  (324.97,161.78) -- (324.35,248.02) ;
\draw    (396.33,161.3) -- (396.33,247.54) ;
\draw  [fill={rgb, 255:red, 0; green, 0; blue, 0 }  ,fill opacity=1 ] (175.91,247.7) .. controls (175.91,246.86) and (176.6,246.18) .. (177.46,246.18) .. controls (178.32,246.18) and (179.02,246.86) .. (179.02,247.7) .. controls (179.02,248.55) and (178.32,249.23) .. (177.46,249.23) .. controls (176.6,249.23) and (175.91,248.55) .. (175.91,247.7) -- cycle ;
\draw  [fill={rgb, 255:red, 0; green, 0; blue, 0 }  ,fill opacity=1 ] (280.21,247.42) .. controls (280.21,246.58) and (280.91,245.9) .. (281.77,245.9) .. controls (282.63,245.9) and (283.33,246.58) .. (283.33,247.42) .. controls (283.33,248.27) and (282.63,248.95) .. (281.77,248.95) .. controls (280.91,248.95) and (280.21,248.27) .. (280.21,247.42) -- cycle ;
\draw  [fill={rgb, 255:red, 0; green, 0; blue, 0 }  ,fill opacity=1 ] (394.96,247.7) .. controls (394.96,246.86) and (395.66,246.18) .. (396.52,246.18) .. controls (397.38,246.18) and (398.07,246.86) .. (398.07,247.7) .. controls (398.07,248.55) and (397.38,249.23) .. (396.52,249.23) .. controls (395.66,249.23) and (394.96,248.55) .. (394.96,247.7) -- cycle ;
\draw  [fill={rgb, 255:red, 0; green, 0; blue, 0 }  ,fill opacity=1 ] (465.91,162.18) .. controls (465.91,161.34) and (466.6,160.65) .. (467.46,160.65) .. controls (468.32,160.65) and (469.02,161.34) .. (469.02,162.18) .. controls (469.02,163.03) and (468.32,163.71) .. (467.46,163.71) .. controls (466.6,163.71) and (465.91,163.03) .. (465.91,162.18) -- cycle ;
\draw  [fill={rgb, 255:red, 0; green, 0; blue, 0 }  ,fill opacity=1 ] (465.91,247.7) .. controls (465.91,246.86) and (466.6,246.18) .. (467.46,246.18) .. controls (468.32,246.18) and (469.02,246.86) .. (469.02,247.7) .. controls (469.02,248.55) and (468.32,249.23) .. (467.46,249.23) .. controls (466.6,249.23) and (465.91,248.55) .. (465.91,247.7) -- cycle ;
\draw  [fill={rgb, 255:red, 0; green, 0; blue, 0 }  ,fill opacity=1 ] (175.91,162.18) .. controls (175.91,161.34) and (176.6,160.65) .. (177.46,160.65) .. controls (178.32,160.65) and (179.02,161.34) .. (179.02,162.18) .. controls (179.02,163.03) and (178.32,163.71) .. (177.46,163.71) .. controls (176.6,163.71) and (175.91,163.03) .. (175.91,162.18) -- cycle ;
\draw  [fill={rgb, 255:red, 0; green, 0; blue, 0 }  ,fill opacity=1 ] (247.47,162.18) .. controls (247.47,161.34) and (248.17,160.65) .. (249.03,160.65) .. controls (249.89,160.65) and (250.58,161.34) .. (250.58,162.18) .. controls (250.58,163.03) and (249.89,163.71) .. (249.03,163.71) .. controls (248.17,163.71) and (247.47,163.03) .. (247.47,162.18) -- cycle ;
\draw  [fill={rgb, 255:red, 0; green, 0; blue, 0 }  ,fill opacity=1 ] (353.06,162.18) .. controls (353.06,161.34) and (353.76,160.65) .. (354.62,160.65) .. controls (355.48,160.65) and (356.18,161.34) .. (356.18,162.18) .. controls (356.18,163.03) and (355.48,163.71) .. (354.62,163.71) .. controls (353.76,163.71) and (353.06,163.03) .. (353.06,162.18) -- cycle ;
\draw    (87.62,247.7) -- (179.02,247.7) ;
\draw    (86.33,162.18) -- (175.91,162.18) ;
\draw    (98.22,247.9) -- (98.22,80.18) ;
\draw  [fill={rgb, 255:red, 0; green, 0; blue, 0 }  ,fill opacity=1 ] (96.91,162.18) .. controls (96.91,161.34) and (97.6,160.65) .. (98.46,160.65) .. controls (99.32,160.65) and (100.02,161.34) .. (100.02,162.18) .. controls (100.02,163.03) and (99.32,163.71) .. (98.46,163.71) .. controls (97.6,163.71) and (96.91,163.03) .. (96.91,162.18) -- cycle ;
\draw  [fill={rgb, 255:red, 0; green, 0; blue, 0 }  ,fill opacity=1 ] (96.91,248.18) .. controls (96.91,247.34) and (97.6,246.65) .. (98.46,246.65) .. controls (99.32,246.65) and (100.02,247.34) .. (100.02,248.18) .. controls (100.02,249.03) and (99.32,249.71) .. (98.46,249.71) .. controls (97.6,249.71) and (96.91,249.03) .. (96.91,248.18) -- cycle ;
\draw    (396.25,162.02) -- (555.82,162.02) ;
\draw    (395.25,248.02) -- (556.82,248.02) ;
\draw    (81.33,86.18) -- (556.91,86.18) ;
\draw    (177.46,160.65) -- (177.46,79.94) ;
\draw    (249.03,163.71) -- (249.03,81.99) ;
\draw  [dash pattern={on 4.5pt off 4.5pt}]  (324.97,161.78) -- (324.97,80.06) ;
\draw  [dash pattern={on 4.5pt off 4.5pt}]  (396.25,162.02) -- (396.25,80.3) ;
\draw    (467.46,163.71) -- (467.46,81.99) ;
\draw [color={rgb, 255:red, 208; green, 2; blue, 27 }  ,draw opacity=1 ][line width=1.5]    (389.73,86.13) -- (354.62,162.18) ;
\draw    (551.15,81.69) -- (551.15,252.93) ;
\draw [color={rgb, 255:red, 208; green, 2; blue, 27 }  ,draw opacity=1 ][line width=1.5]    (355.52,161.37) -- (280.73,247.47) ;
\draw  [fill={rgb, 255:red, 0; green, 0; blue, 0 }  ,fill opacity=1 ] (96.91,86.18) .. controls (96.91,85.34) and (97.6,84.65) .. (98.46,84.65) .. controls (99.32,84.65) and (100.02,85.34) .. (100.02,86.18) .. controls (100.02,87.03) and (99.32,87.71) .. (98.46,87.71) .. controls (97.6,87.71) and (96.91,87.03) .. (96.91,86.18) -- cycle ;
\draw  [fill={rgb, 255:red, 0; green, 0; blue, 0 }  ,fill opacity=1 ] (175.91,86.18) .. controls (175.91,85.34) and (176.6,84.65) .. (177.46,84.65) .. controls (178.32,84.65) and (179.02,85.34) .. (179.02,86.18) .. controls (179.02,87.03) and (178.32,87.71) .. (177.46,87.71) .. controls (176.6,87.71) and (175.91,87.03) .. (175.91,86.18) -- cycle ;
\draw  [fill={rgb, 255:red, 0; green, 0; blue, 0 }  ,fill opacity=1 ] (247.91,86.18) .. controls (247.91,85.34) and (248.6,84.65) .. (249.46,84.65) .. controls (250.32,84.65) and (251.02,85.34) .. (251.02,86.18) .. controls (251.02,87.03) and (250.32,87.71) .. (249.46,87.71) .. controls (248.6,87.71) and (247.91,87.03) .. (247.91,86.18) -- cycle ;
\draw  [fill={rgb, 255:red, 0; green, 0; blue, 0 }  ,fill opacity=1 ] (386.91,86.18) .. controls (386.91,85.34) and (387.6,84.65) .. (388.46,84.65) .. controls (389.32,84.65) and (390.02,85.34) .. (390.02,86.18) .. controls (390.02,87.03) and (389.32,87.71) .. (388.46,87.71) .. controls (387.6,87.71) and (386.91,87.03) .. (386.91,86.18) -- cycle ;
\draw  [fill={rgb, 255:red, 0; green, 0; blue, 0 }  ,fill opacity=1 ] (465.91,86.18) .. controls (465.91,85.34) and (466.6,84.65) .. (467.46,84.65) .. controls (468.32,84.65) and (469.02,85.34) .. (469.02,86.18) .. controls (469.02,87.03) and (468.32,87.71) .. (467.46,87.71) .. controls (466.6,87.71) and (465.91,87.03) .. (465.91,86.18) -- cycle ;
\draw  [fill={rgb, 255:red, 0; green, 0; blue, 0 }  ,fill opacity=1 ] (549.91,86.18) .. controls (549.91,85.34) and (550.6,84.65) .. (551.46,84.65) .. controls (552.32,84.65) and (553.02,85.34) .. (553.02,86.18) .. controls (553.02,87.03) and (552.32,87.71) .. (551.46,87.71) .. controls (550.6,87.71) and (549.91,87.03) .. (549.91,86.18) -- cycle ;
\draw  [fill={rgb, 255:red, 0; green, 0; blue, 0 }  ,fill opacity=1 ] (549.36,162.18) .. controls (549.36,161.34) and (550.05,160.65) .. (550.91,160.65) .. controls (551.77,160.65) and (552.47,161.34) .. (552.47,162.18) .. controls (552.47,163.03) and (551.77,163.71) .. (550.91,163.71) .. controls (550.05,163.71) and (549.36,163.03) .. (549.36,162.18) -- cycle ;
\draw  [fill={rgb, 255:red, 0; green, 0; blue, 0 }  ,fill opacity=1 ] (549.36,248.18) .. controls (549.36,247.34) and (550.05,246.65) .. (550.91,246.65) .. controls (551.77,246.65) and (552.47,247.34) .. (552.47,248.18) .. controls (552.47,249.03) and (551.77,249.71) .. (550.91,249.71) .. controls (550.05,249.71) and (549.36,249.03) .. (549.36,248.18) -- cycle ;

\draw (163.16,250.18) node [anchor=north west][inner sep=0.75pt]  [font=\fontsize{0.59em}{0.71em}\selectfont,xscale=0.8,yscale=0.8]  {$x_{m( n) -3/2}$};
\draw (448.45,250.18) node [anchor=north west][inner sep=0.75pt]  [font=\fontsize{0.59em}{0.71em}\selectfont,xscale=0.8,yscale=0.8]  {$x_{m( n) +5/2}$};
\draw (377.19,250.9) node [anchor=north west][inner sep=0.75pt]  [font=\fontsize{0.59em}{0.71em}\selectfont,xscale=0.8,yscale=0.8]  {$x_{m( n) +3/2}$};
\draw (110,225.73) node [anchor=north west][inner sep=0.75pt]  [xscale=0.8,yscale=0.8]  {$\mathfrak{C}_{m( n) -2}$};
\draw (405.33,225.73) node [anchor=north west][inner sep=0.75pt]  [xscale=0.8,yscale=0.8]  {$\mathfrak{C}_{m( n) +2}$};
\draw (242.33,224.73) node [anchor=north west][inner sep=0.75pt]  [xscale=0.8,yscale=0.8]  {$D_{L}^{n}$};
\draw (346.67,228.07) node [anchor=north west][inner sep=0.75pt]  [xscale=0.8,yscale=0.8]  {$D_{R}^{n}$};
\draw (310.33,136.73) node [anchor=north west][inner sep=0.75pt]  [xscale=0.8,yscale=0.8]  {$D_{L}^{n+1}$};
\draw (384.33,138.07) node [anchor=north west][inner sep=0.75pt]  [xscale=0.8,yscale=0.8]  {$D_{R}^{n+1}$};
\draw (64.8,234.2) node [anchor=north west][inner sep=0.75pt]  [xscale=0.8,yscale=0.8]  {$t^{n}$};
\draw (62.8,147.53) node [anchor=north west][inner sep=0.75pt]  [xscale=0.8,yscale=0.8]  {$t^{n+1}$};
\draw (186,140.4) node [anchor=north west][inner sep=0.75pt]  [xscale=0.8,yscale=0.8]  {$\mathfrak{C}_{m( n) -2}$};
\draw (474.67,137.73) node [anchor=north west][inner sep=0.75pt]  [xscale=0.8,yscale=0.8]  {$\mathfrak{C}_{m( n) +2}$};

\end{tikzpicture}   
    \caption{\small{
    Case B. 
    Dashed lines indicate that the corresponding cell boundary is removed.
    }}
    \label{fig:moving-mesh-B-def-of-rho}
\end{figure}
 
\medskip
\boxed{\textbf{Case C:}\ m(n+1) = m(n)- 1}\quad
This case is possible only if $\alpha_{n,\DT} < 0$ and is peculiar
to case \textbf{B}:
\begin{align*}
    t^{n+1}:\qquad &x_{m(n+1)-3/2}\,,\quad\quad\quad {\bord \xi^{n+1}}\,,  \qquad x_{m(n+1)+3/2}\,,
    \qquad x_{m(n+1)+5/2}\,;\\
    t^n:\qquad &x_{m(n)-5/2}\,, \quad\quad x_{m(n)-3/2}\,,\qquad\quad {\bord \xi^n}\,,\qquad\quad \qquad x_{m(n)+3/2}  \,.
\end{align*}
The following identities hold:
\begin{equation}
\label{caseC-id_DX}
\Delta^{n+1}_L= \Delta^n_L +\alpha_{n,\DT} \DT + \DX
\,,\qquad 
\DX + \Delta^{n+1}_R= \Delta^n_R - \alpha_{n,\DT} \DT
\,.    \end{equation}

\subsection{The numerical fluxes}
In this subsection, we introduce the numerical fluxes employed in the scheme.

Define
\begin{equation}\label{hpm}
    h^\pm(a,b) : = \begin{cases}
    \min_{[a,b]} (\pm f), & \mbox { if }a\le b,\\
    \max_{[b,a]} (\pm f), & \mbox { if }a \ge b
    \end{cases}
\end{equation}
and, for $\alpha\in\eR$, 
\begin{equation}\label{h0}
h_0(a,b,\alpha) = \begin{cases}
    f(b) - \alpha b = - b \left(\alpha - v(b) \right), &\mbox{ if }\alpha \ge v(b),\\
     - f(a) - \alpha a = a \left( -\alpha - v(a)  \right), &\mbox{ if }\alpha \le - v(a)\\
    0, & \mbox{ otherwise}\,,
\end{cases}    
\end{equation}
or equivalently 
\begin{align} 
      h_0(a,b,\alpha) & = -b \,[v(b) -\alpha ]_-  + a \, [v(a)+\alpha]_-,   \label{h0-bis}
\end{align}
where $[x]_- = \max\{-x,0\}$. 

The maps $h^\pm$ represent the Godunov fluxes at a vertical boundary of a cell, corresponding to the fluxes $\pm f$ respectively, while  
the map $h_0:[0,1]\times[0,1]\times\eR$ represents the flux along the possibly non-vertical boundary cell with slope $\alpha$.

Notice that
\begin{equation*}
    h^-(a,b)  = - 
     \begin{cases}
    \max_{[a,b]} f, & \mbox { if }a\le b,\\
    \min_{[b,a]} f, & \mbox { if }a\ge b,
    \end{cases}
\end{equation*}  
and hence $h^-(a,b)  = - h^+ (b,a)$.

Moreover, the maps $h^\pm$, $h_0$ are Lipschitz continuous with Lipschitz constant
\begin{equation}\label{Lip-num-fluxes}
    Lip(h^\pm) \le  \|f'\|_\infty \,,\qquad
    Lip(h_0) \le   \|f'\|_\infty + |\alpha|\,.
\end{equation}

We stress that the numerical fluxes $h^\pm$, $h_0$ are based on the exact Riemann solver at the interface. See Section \ref{sec:riemann}. In the following lemma, we provide some monotonicity properties of the numerical fluxes and the value of $h_0$ when its arguments $(a,b,\alpha)$ 
satisfy the Rankine-Hugoniot condition \eqref{eq:Rankine-hugoniot-main} at the flux interface. 

\begin{lemma} \label{rem:mon-fluxes} 
For any given $a$ and $b$ in $[0,1]$ the following hold
\begin{itemize}
\item[{\bf (i)}] $h^\pm(a,a)=\pm f(a)$, 
\item[{\bf (ii)}] $h^\pm(a,b)$ and $h_0(a,b, \alpha)$ are nondecreasing with respect to $a$ and nonincreasing with respect to $b$, for all $\alpha\in \eR$,
\item[{\bf (iii)}] $h_0(a,b,0) = 0$,
\item[{\bf (iv)}] for all $\alpha$, 
\begin{equation}\label{eq:h0cc}
 h_0(a,a,\alpha) =    \begin{cases}
        -\alpha a + f(a), \quad & \mbox{ if } \alpha\ge v(a) \,,\\
        -\alpha a - f(a), \quad & \mbox{ if } \alpha\le -v(a) \,,\\
        0, \quad & \mbox{ otherwise.}
    \end{cases}
\end{equation}
\item[{\bf (v)}]  If \begin{equation}\label{hyp:h0-germ}
     f(b)+f(a) = \alpha (b-a),
 \end{equation}
then 
\begin{equation}\label{eq:h0-germ}
     h_0(a,b,\alpha)  = f(b)-\alpha b = -\alpha a - f(a)\,. 
\end{equation}
\end{itemize}
\end{lemma}
\begin{proof}
All of the properties {\bf (i)}--{\bf (iv)} are direct consequences of the definitions in \eqref{hpm}--\eqref{h0-bis} and the fact that the map 
$\rho\mapsto \rho[v(\rho) + const.]_-$ is nondecreasing, being the composition of two nonincreasing maps.

Concerning {\bf (v)}, 
by \eqref{h0-bis} and \eqref{hyp:h0-germ}
\begin{align*}
      h_0(a,b,\alpha) 
    & = - [\alpha b -f(b)]_+  +  [\alpha a +f(a)]_-  \\
    &= - [\alpha b -f(b)]_+  +  [\alpha b - f(b)]_- \\
    & = - \alpha b + f(b),
\end{align*}
where we used the identity $[x]_+ - [x]_- = x$. Hence we get the first equality in \eqref{eq:h0-germ}; 
the second one follows immediately from \eqref{hyp:h0-germ}. 
\end{proof}

  
\subsection{The numerical scheme}
In this subsection we define a family of approximate solution to \eqref{eq:main}. We assume the notation of Definition~\ref{def:the-mesh} 
and the CFL condition \eqref{eq:CFL-cond} for the space and time steps $\DX$ and $\DT$. 

The approximate solution $\rho^{\DX}(x,t)$, at time $t^n$, takes the form
\begin{equation}\label{eq:rhoDX-at-tn}
    \rho^\DX(x,t^n) = \rho^n_j \qquad x\in \CCC^n_j\,,\qquad j\not = m(n)\,,\quad |j|\le \frac 1 \DX
    \,,\quad n\ge 0
\end{equation}
for suitable values $\rho^n_j$ to be defined and $j\in \eZ$.
In addition, we take into account the boundary conditions by considering two extra values for the density:
\begin{equation}\label{eq:boundary-values-0}
    \rho^n_{\pm(1 + 1/\DX)}=0\,,\qquad n\ge 0
\end{equation}
on the two extra cells
$$
C_{-\frac 1\DX-1}:= \left[x_{-\frac 1\DX-\frac32}, x_{-\frac 1 \DX-\frac12}\right)\,,\qquad  C_{1+ \frac 1\DX}:= \left[x_{\frac 1\DX+\frac12}, x_{\frac 1\DX+\frac32}\right)\,.
$$

For convenience, we will use the notation
\begin{equation}\label{rho-n-L-R}
\rho^n_L := \rho^n_{m(n)-1}\,,\qquad \rho^n_R := \rho^n_{m(n)+1}  \,.    
\end{equation}
\medskip
Let's proceed with the definition of the values $\rho^n_j$.

\boxed{\textbf{Initialization:}}
At $t=0$ the initial data $\rho_0(x)\in L^\infty(-1,1)$ 
is discretized by the local average on the cell:
\begin{equation}\label{eq:rho-0-j}
    \rho^0_j = \frac 1 {\DX^0_j
    } \int_{\CCC^0_j} \rho_0(x)\, dx\,,\qquad 
    j\not = m(0)\,, 
\end{equation}
with $\CCC^0_j$, $\DX^0_j$ as in \eqref{def:new-cells}, \eqref{def:cell-size} for $n=0$.

\boxed{\textbf{Iteration step:}} Assume that the $\rho^n_j$ are known for some $n\ge 0$ and $|j|\le 2^N$, $j\not = m(n)$, 
and that \eqref{eq:boundary-values-0} holds.

Having set $\rho^\DX(x,t^n)$ as in \eqref{eq:rhoDX-at-tn} and recalling $\xi^\DT(t)$ from \eqref{def:xi-DT}, the solution $\rho^\DX$ is prolonged for $t\in (t^n,t^{n+1})$ 
as the \textit{exact} solution of the problem 
\begin{equation}\label{eq:main-approx}
	\begin{cases}
		\pt\rho + \px \left\{\sign(x-\xi^\DT(t))f(\rho)\right\} = 0,\\
		\rho(x,t^n) = \rho^\DX(x,t^n) 
	\end{cases}
\end{equation}
on the set $\left(-\frac {3} 2\DX -1,1+\frac {3} 2 \DX\right)$.

Notice that, thanks to \eqref{eq:boundary-values-0}, the following condition is satisfied:
\begin{equation}\label{eq:bc}
    \rho^\DX(t,\pm(1+\DX/2))\in [0,{{\blue \argmax}}_{[0,1]} f]
\end{equation}
and this will serve as the boundary condition for the original problem. See \cite{DeboraGoatinRosini2014,El-Khatib2013Weak}. 

Now, we define the values
$$
\rho^{n+1}_j\,,\qquad j\in\{-\frac1\DX,\ldots, \frac1\DX\}\setminus\{m(n+1)\}
$$
in the Cases {\bf A}, {\bf B} and {\bf C}, as follows. 

\medskip
\boxed{\textbf{Case A:}\ m(n)=m(n+1)} \quad 
Here, with the notation of \eqref{def:cell-size}, we set
\begin{equation}\label{eq:rho-j-n+1}
    \rho^{n+1}_j := \frac 1 {\DX^{n+1}_j}
    \int_{\CCC^{n+1}_j} \rho^\DX(x,t^{n+1}-)\, dx\,,\qquad j\not = m(n+1)\,.
\end{equation}
In particular, we have the following values 
\begin{itemize}
    \item For $|j-m(n)|\ge 2$
    the boundary of the cells are vertical and the $\rho_j^{n+1}$ are defined in a standard way: 
    for $-1/\DX \le j\le m(n+1)- 2$,
\begin{equation}\label{eq:standard-marching-minus}
   \rho^{n+1}_j = \rho^n_j - \frac{\DT}{\DX} \cdot
        \big[h^-(\rho^n_j,\rho^n_{j+1})- h^-(\rho^n_{j-1}, \rho^n_{j})\big]
\end{equation}
while for $m(n+1)+ 2 \le j \le 1/\DX$:
\begin{equation}\label{eq:standard-marching-plus}
      \rho^{n+1}_j = \rho^n_j - \frac{\DT}{\DX}\cdot 
               \big[h^+(\rho^n_j,\rho^n_{j+1})- h^+(\rho^n_{j-1}, \rho^n_{j})\big] \,.
\end{equation}
 Notice that for the extreme values of the indexes $j = \pm 1/\DX$ we use the extended values \eqref{eq:boundary-values-0}; for instance, if $j=1/\DX$ then \eqref{eq:standard-marching-plus} reduces to
\begin{equation*}
  \rho^{n+1}_{1/\DX} = \rho^n_{1/\DX}  - \frac{\DT}{\DX}\cdot 
               \big[h^+(\rho^n_{1/\DX},0)- h^+(\rho^n_{{1/\DX} -1}, \rho^n_{{1/\DX} })\big]\,.
\end{equation*}
%
%

\medskip
\item For $j=m(n)-1$
we consider the trapezoid in the $(x,t)$-plane with corners 
$$
\left(x_{m(n)-3/2},t^n\right)
\,,\qquad \left(\xi^n,t^n\right)\,,\qquad \left(\xi^{n+1},t^{n+1}\right)\,,\qquad \left(x_{m(n)-3/2},t^{n+1}\right)
\,;
$$
It can be observed that, using the notation given in equation \eqref{def:DnLR}, the lower basis is $D^n_L$ and the upper basis is $C_{m(n)-1}^n = D^{n+1}_L$. 
%
Then we integrate $\rho^\DX(x,t)$ over the trapezoid 
to get
\begin{align*}
    &\int_{D^{n+1}_L} \rho^\DX(x,t^{n+1}-)\,dx - \int_{D^n_L} \rho^\DX(x,t^{n}+)\,dx  \\
    &\qquad \qquad + ~
   \int_{t^n}^{t^{n+1}} h_0\left(\rho^n_L,\rho^n_R,\alpha_{n,\DT}\right)\, dt - 
    \int_{t^n}^{t^{n+1}} h^-\left(\rho^n_{m(n)-2},\rho^n_L\right)\, dt=0\,,
\end{align*}
that is
\begin{align}\label{eq:rho_n+1_L-A}
   &  \Delta^{n+1}_L 
   \, \rho^{n+1}_L =
   \Delta^n_L 
   \,\rho^{n}_L 
     ~-~ \DT \left[ h_0\left(\rho^n_L,\rho^n_R,\alpha_{n,\DT}\right) - 
 h^-\left(\rho^n_{m(n)-2},\rho^n_L\right)\right],
\end{align}
where $\Delta^n_L$ is defined at \eqref{def:cell-size}. By recalling \eqref{eq:lenght-cellsLR-A}, the formula here above can rewritten as
\begin{equation*}
    \rho^{n+1}_L =
   \rho^{n}_L 
     ~-~ \frac\DT{\Delta^{n+1}_L } \left[ h_0\left(\rho^n_L,\rho^n_R,\alpha_{n,\DT}\right) - 
 h^-\left(\rho^n_{m(n)-2},\rho^n_L\right) + \alpha_{n,\DT}  \rho^{n}_L \right]\,.
\end{equation*}


\medskip
\item For $j=m(n)+1$, we proceed similarly as the previous case and define
\begin{align}\label{eq:rho_n+1_R-A}
\Delta^{n+1}_R \rho^{n+1}_R
&= {\Delta^n_R} 
   ~\rho^{n}_{R}  -~ {\DT}
    \left[h^+\left(\rho^n_R
    ,\rho^n_{m(n)+2}\right) - h_0\left(\rho^n_L,\rho^n_R,\alpha_{n,\DT}\right) 
 \right] , 
\end{align}
with $\Delta^{n+1}_R = \Delta^n_R - \alpha_{n,\DT}\DT$, see \eqref{eq:lenght-cellsLR-A}$_2$. Consequently, we obtain
\begin{equation*}
    \rho^{n+1}_R =
   \rho^{n}_R 
     ~-~ \frac\DT{\Delta^{n+1}_R } \left[
     h^+\left(\rho^n_R
    ,\rho^n_{m(n)+2}\right) - h_0\left(\rho^n_L,\rho^n_R,\alpha_{n,\DT}\right) - \alpha_{n,\DT}  \rho^{n}_R \right]\,.
\end{equation*}
\end{itemize}

\medskip
\boxed{\textbf{Case B:}\ m(n+1)=m(n)+1}\quad  
We focus on the region in the time interval $[t^n,t^{n+1}]$ which is delimited by the cell points
\begin{equation*}
    x_{m(n)-3/2}=  x_{m(n+1)-5/2}\,,\qquad   x_{m(n)+5/2}=  x_{m(n+1)+3/2}\,.
\end{equation*}
At time $t^n$ the region above corresponds to three cells and three corresponding values of the density
$$
D^n_L\,,\ \rho^n_L; 
\qquad D^n_R\,,\ \rho^n_R
;\qquad 
\CCC^n_{m(n)+2}\,,\ \rho^n_{m(n)+2}\,.
$$
On the other hand, at time $t^{n+1}$ the cells and values in the region are:
\begin{align}\label{eq:def-rho-caseB}
\CCC^{n+1}_{m(n+1)-2}\,,~ \rho^{n+1}_{m(n+1)-2};
\qquad  D^{n+1}_L\,,~ \rho^{n+1}_L
;\qquad D^{n+1}_R\,,~ \rho^{n+1}_R
\,.    
\end{align}
Now we set
\begin{equation}\label{def:left-values-caseB}
    \rho^{n+1}_{m(n+1)-2} =  \rho^{n+1}_L
    := \bar \rho,
\end{equation}
where the value $\bar \rho$ is defined by the integration on the trapezoid delimited by $x_{m(n)-3/2}=  x_{m(n+1)-5/2}$ and the segment connecting $(\xi^n,t^n)$ to $(\xi^{n+1},t^{n+1})$, as follows:
\begin{align}\label{eq:march_caseB_L}
  \left(\DX + \Delta^{n+1}_L 
\right) 
   ~\bar \rho ~=~
  \Delta^n_L 
   ~\rho^{n}_L 
    -~ \DT \left[ h_0\left(\rho^n_L,\rho^n_R
    ,\alpha_{n,\DT}
    \right) - 
 h^-\left(\rho^n_{m(n)-2},\rho^n_L
 \right)\right]\,.
\end{align}
By using the identity \eqref{caseB-id_DX}$_1$, we find that
\begin{equation*}
  \rho^{n+1}_L = \rho^{n}_L - \frac{\DT}{\DX + \Delta^{n+1}_L}  
        \left[ h_0\left(\rho^n_L,\rho^n_R
    ,\alpha_{n,\DT}
    \right) - 
 h^-\left(\rho^n_{m(n)-2},\rho^n_L
 \right) + \alpha_{n,\DT}\rho^{n}_L\right]\,.
\end{equation*}

For the definition of $\rho^{n+1}_R$, we proceed similarly by integration on the trapezoid on the right of the segment connecting 
$(\xi^n,t^n)$ to $(\xi^{n+1},t^{n+1})$: we obtain that it is defined by the identity
\begin{align}
\Delta^{n+1}_R \rho^{n+1}_R 
= &
   \Delta^{n}_R
   ~\rho^{n}_R
   + \DX  ~\rho^{n}_{m(n)+2}
   -~ {\DT} 
    \left[h^+\left(\rho^n_{m(n)+2},\rho^n_{m(n)+3}\right) - h_0\left(\rho^n_L
    ,\rho^n_R 
    ,\alpha_{n,\DT}\right) 
 \right]\,.   \label{eq:rho_n+1_R-B} 
\end{align}
By means of the identity \eqref{caseB-id_DX}$_2$ we rewrite \eqref{eq:rho_n+1_R-B} as follows:
\begin{align*}
      \rho^{n+1}_R &  = \frac{\Delta^{n+1}_R -\DX}
      {\Delta^{n+1}_R }\rho^{n}_R + \frac{\DX}{\Delta^{n+1}_R }
   ~\rho^{n}_{m(n)+2}\\
   & \qquad  - \frac{\DT}{\Delta^{n+1}_R}  
        \left[
        h^+\left(\rho^n_{m(n)+2},\rho^n_{m(n)+3}\right) - h_0\left(\rho^n_L
    ,\rho^n_R 
    ,\alpha_{n,\DT}\right) - \alpha_{n,\DT}\rho^{n}_R\right]\,.
\end{align*}
This completes the definition of the three values in \eqref{eq:def-rho-caseB}, that is, for 
$\rho^{n+1}_j$  with $j=m(n+1)-2$ and $j=m(n+1)\pm 1$.

On the other hand, the values $\rho^{n+1}_j$ for $j\le m(n+1)-3$ and $j\ge m(n+1)+2$ are given by the standard formulas \eqref{eq:standard-marching-minus}, \eqref{eq:standard-marching-plus} respectively. 

This completes 
the definition of the values $\rho^{n+1}_j$ with $j\not = m(n+1)$ 
for Case B.

\medskip
\boxed{\textbf{Case C:}\ m(n+1)=m(n)-1}\quad 
In this case, we focus on the region in the time interval $[t^n,t^{n+1}]$, which is delimited by the cell points
\begin{equation*}
    x_{m(n)-5/2}=  x_{m(n+1)-3/2}\,,\qquad   x_{m(n)+3/2}=  x_{m(n+1)+5/2}\,.
\end{equation*}
At time $t^n$, the region above corresponds to three cells and three corresponding values:
$$
\CCC^n_{m(n)-2}\,,\ \rho^n_{m(n)-2}\, ;\qquad
D^n_L\,,\ \rho^n_L \,; 
\qquad D^n_R\,,\ \rho^n_R \,. 
$$
On the other hand, at time $t^{n+1}$ the cells and values in the region are:
\begin{align}\label{eq:def-rho-caseC}
 D^{n+1}_L\,,~ \rho^{n+1}_L
;\qquad D^{n+1}_R\,,~ \rho^{n+1}_R \,;
\qquad \CCC^{n+1}_{m(n+1)+2}\,,~ \rho^{n+1}_{m(n+1)+2}\,.    
\end{align}
The marching formulas which allow us to compute the values of $\rho^{n+1}_L$, $\rho^{n+1}_R$, $\rho^{n+1}_{m(n+1)+2}$, are mirror images of the ones in \eqref{def:left-values-caseB}--\eqref{eq:rho_n+1_R-B}, as the only difference between the two cases is the sign of $\alpha_{n,\DT}$ (positive in Case \textbf{B}, negative in Case \textbf{C}). For the reader's convenience, we write them explicitly.  
\begin{align} \label{eq:rho_n+1_L-C} 
\Delta^{n+1}_L \rho^{n+1}_L 
= &
   \Delta^{n}_L
   ~\rho^{n}_L
   + \DX  ~\rho^{n}_{m(n)-2}
   -~ {\DT} 
    \left[h_0\left(\rho^n_L
    ,\rho^n_R
    ,\alpha_{n,\DT}\right) -h^-\left( \rho^n_{m(n)-3},\rho^n_{m(n)-2}\right) 
 \right]\,.   
\end{align}
Now, we set
\begin{equation}\label{def:right-values-caseC}
     \rho^{n+1}_R = \rho^{n+1}_{m(n+1)+2} 
    := \bar \rho,
\end{equation}
where the value $\bar \rho$ is defined by
\begin{align*}
\left(\DX + \Delta^{n+1}_R 
\right) 
   ~\bar \rho ~=~
  \Delta^n_R 
   ~\rho^{n}_R 
    -~ \DT \left[
 h^+\left(\rho^n_R, \rho^n_{m(n)+2},
 \right)-h_0\left(\rho^n_L,\rho^n_R,\alpha_{n,\DT} \right)  \right]\,.
\end{align*}
 The values $\rho^{n+1}_j$ for $j\le m(n+1)-2$ and $j\ge m(n+1)+3$ are given by the standard formulas \eqref{eq:standard-marching-minus}, \eqref{eq:standard-marching-plus} respectively. 

This completes the definition of the values $\rho^{n+1}_j$ with $j\not = m(n+1)$ for Case \textbf{C}.

\section{Analysis of the scheme} \label{sec:AnalysisOfTheScheme} 
In this section, we analyze the scheme and prove relevant properties that provide a strong basis to establish its convergence to the entropy solution defined in Section \ref{sec:ProblemSettings}. For convenience, we rewrite the equation \eqref{eq:main} here as:
\begin{equation}\label{eq:main_S4}
	\begin{cases}
		\pt\rho + \px \left\{\sign(x-\xi(t))f(\rho)\right\} = 0,\\
		\rho(\cdot,0) = \rho_0 \,,\\
		\rho(t,\pm1)\in [0,{{\blue \argmax}}_{[0,1]} f]
	\end{cases}
\end{equation}
with $\xi(t)$ satisfying \eqref{eq:hyp-on-xi}.

Let's start by introducing a convenient expression of the scheme that will be useful during the analysis.
\medskip
We define
\begin{align}
    H^\pm(a,b,c;\DT,\DX) &= b -
    \frac{\DT}{\DX} \big[h^\pm(b,c)- h^\pm(a, b)\big]\,,\label{Hpm}\\[2mm]
    H^{AB}_L (a,b,c;\DT,\Delta_L, \alpha) &= \frac{\Delta_L}{\Delta_L + \alpha\DT} b - \frac{\DT}{\Delta_L + \alpha\DT} 
     \left(h_0\left(b,c,\alpha\right) - 
 h^-\left(a,b\right)\right)
 \label{HA_L}
      \\[2mm] 
     \label{HA_R}
      H^{AC}_R (a,b,c;\DT,\Delta_R, \alpha) &= 
      \frac{\Delta_R}{\Delta_R - \alpha\DT} b - \frac{\DT}{\Delta_R - \alpha\DT} 
     \left( 
 h^+\left(b,c\right)- h_0\left(a,b,\alpha\right) \right)\,,
\end{align}
 and
 \begin{align}
 \label{HC_L}
      H^C_L (a,b,c, d;\DT,\DX, \Delta_L, \alpha) &= 
      \frac{\Delta_L}{\Delta_L + \DX + \alpha\DT} b + \frac{\DX}{\Delta_L + \DX + \alpha\DT} c \\[2mm]\nonumber
      &\qquad - \frac{\DT}{\Delta_L + \DX +\alpha\DT} 
     \left( 
 h_0\left(c,d,\alpha\right) -h^-\left(a,b\right)\right)\,,
     \\[2mm] \label{HB_R}
      H^B_R (a,b,c, d;\DT,\DX, \Delta_R, \alpha) &= 
      \frac{\Delta_R}{\Delta_R + \DX - \alpha\DT} b + \frac{\DX}{\Delta_R + \DX - \alpha\DT} c \\[2mm]\nonumber
      &\qquad - \frac{\DT}{\Delta_R + \DX - \alpha\DT} 
     \left( 
 h^+\left(c,d\right)- h_0\left(a,b,\alpha\right) \right)\,.
\end{align}
Then the expressions \eqref{eq:standard-marching-minus}--\eqref{eq:rho_n+1_R-A}, Case \textbf{A}, can be rewritten as

\begin{equation}\label{eq:Hpm}
    \rho^{n+1}_j = \begin{cases} H^-(\rho^n_{j-1},\rho^n_{j},\rho^n_{j+1};\DT,\DX), & - 1/\DX 
    \le j\le m(n+1)- 2,\\[2mm]
    H^+(\rho^n_{j-1},\rho^n_{j},\rho^n_{j+1};\DT,\DX), & m(n+1) + 2 \le j\le 1/\DX
    \,,
    \end{cases}
\end{equation}
and 
\begin{equation}\label{eq:HA_LR}
    \rho^{n+1}_j = \begin{cases} 
    H^{AB}_L(\rho^n_{m(n)-2},\rho^n_L,\rho^n_R; \DT, \Delta^n_L,\alpha_{n,\DT}), & j=m(n+1)-1, \\[2mm]
    H^{AC}_R(\rho^n_L,\rho^n_R, \rho^n_{m(n)+2}; \DT, \Delta^n_R,\alpha_{n,\DT}), & j=m(n+1)+1\,. 
    \end{cases}
\end{equation}

The corresponding expressions for the Case \textbf{B} write as 
\begin{equation}\label{eq:HpmB}
    \rho^{n+1}_j = \begin{cases} H^-(\rho^n_{j-1},\rho^n_{j},\rho^n_{j+1};\DT,\DX), & - \frac1\DX 
    \le j\le m(n+1)- 3,\\[2mm]
    H^+(\rho^n_{j-1},\rho^n_{j},\rho^n_{j+1};\DT,\DX), & m(n+1) + 2 \le j\le \frac1\DX \,,
    \end{cases}
\end{equation}
and 
\begin{equation}\label{eq:HA_LRB}
    \rho^{n+1}_j = \begin{cases} 
    H^{AB}_L(\rho^n_{m(n)-2},\rho^n_L,\rho^n_R; \DT, \Delta^n_L,\alpha_{n,\DT}), & j=m(n+1)-2,\\ 
    &\qquad m(n+1)-1, \\[2mm]
    H^B_R(\rho^n_L,\rho^n_R, \rho^n_{m(n)+2},\rho^n_{m(n)+3}; \DT, \DX, \Delta^n_R,\alpha_{n,\DT}), & j=m(n+1)+1\,. 
    \end{cases}
\end{equation}
Finally, for the Case \textbf{C} we have 
\begin{equation}\label{eq:HpmC}
    \rho^{n+1}_j = \begin{cases} H^-(\rho^n_{j-1},\rho^n_{j},\rho^n_{j+1};\DT,\DX), & -2^N\le j\le m(n+1)- 2,\\[2mm]
    H^+(\rho^n_{j-1},\rho^n_{j},\rho^n_{j+1};\DT,\DX), & m(n+1) + 3 \le j\le 2^N\,,
    \end{cases}
\end{equation}
and 
\begin{equation}\label{eq:HA_LRC}
    \rho^{n+1}_j = \begin{cases} 
    H^{C}_L(\rho^n_{m(n)-3},\rho^n_{m(n)-2},\rho^n_L,\rho^n_R; \DT, \DX, \Delta^n_L,\alpha_{n,\DT}), &j= m(n+1)-1, \\[2mm]
    H^{AC}_R(\rho^n_L,\rho^n_R, \rho^n_{m(n)+2}; \DT, \Delta^n_R,\alpha_{n,\DT}),  &j=m(n+1)+1,\\
    &\qquad m(n+1)+2\,.
    \end{cases}
\end{equation}

In the following lemma, we establish some monotonicity properties of the maps defined at \eqref{Hpm}--\eqref{HB_R}.
\begin{lemma}
\label{lem:properties-of-H}
Given that the CFL condition \eqref{eq:CFL-cond} is satisfied, the functions $H^\pm$, $H^{AB}_L$ and $H^{AC}_R$ 
are nondecreasing with respect to their first three arguments $a,b,c$. 

The functions $H^C_L$ and $H^B_R$ 
are nondecreasing with respect of their first four arguments $a,b,c, d$. 
\end{lemma}
\begin{proof} The functions $H^\pm$, $H^{AB}_L$, $H^{AC}_R$, $H^B_R$, and $H^C_L$  are Lipschitz continuous; therefore, we can check their monotonicity properties by computing a.e. the partial derivatives. To keep our presentation as light as possible, we provide detailed estimates only for $H^B_R$, see \eqref{HB_R}, as the other cases are very similar. 

Thanks to Lemma~\ref{rem:mon-fluxes}, the CFL condition \eqref{eq:CFL-cond} and the fact that $\DX\leq\Delta_R\leq 2 \DX$, we have that 
\begin{align*}
    \partial_a H^B_R (a,b,c,d;\DT,\DX, \Delta_R, \alpha) 
    & =  \frac{\DT}{\Delta_R + \DX - \alpha\DT} 
     \partial_a h_0\left(a,b,\alpha\right)\geq 0 \,,\\[2mm] 
    \partial_b H^B_R (a,b,c,d;\DT,\DX, \Delta_R, \alpha) 
    & = \frac{\Delta_R}{\Delta_R + \DX - \alpha\DT} - \frac{\DT}{\Delta_R + \DX - \alpha\DT} 
    \vert\partial_b h_0\left(a,b,\alpha\right)\vert\\
    &\geq \frac{1}{2}\frac{2\Delta_R - \DX}{\Delta_R + \DX - \alpha\DT}>0 \,,\\[2mm] 
    \partial_c H^B_R (a,b,c,d;\DT,\DX, \Delta_R, \alpha) & = \frac{\DX}{\Delta_R + \DX - \alpha\DT} - \frac{\DT}{\Delta_R + \DX - \alpha\DT} 
    \partial_c h^+\left(c,d\right)\\
    &\geq 
    \frac{1}{2}\frac{\DX}{\Delta_R + \DX - \alpha\DT}>0 \,,\\[2mm] 
    \partial_d H^B_R (a,b,c,d;\DT,\DX, \Delta_R, \alpha)  & = \frac{\DT}{\Delta_R + \DX - \alpha\DT} 
     \vert\partial_d h^+\left(c,d\right)\vert\geq 0 \,. 
\end{align*}
\end{proof}

For future use, we define the remainder terms
    \begin{equation}
 \RRR_L(b,\alpha) = \begin{cases}
        0, \quad & \textrm{ if }\alpha\le -v(b) \,,\\
        f(b)+\alpha b, \quad & \textrm{ if } \alpha\in (-v(b),v(b)) \,,\\
        2f(b), \quad & \textrm{ if }\alpha\ge v(b) \,,
    \end{cases}        
    \end{equation}%
and  $\RRR_R(b,\alpha) := 2f(b)-\RRR_L(b,\alpha) $. Observe that the remainder terms are always non negative. %

Let us now collect other useful properties of the maps defined at \eqref{Hpm}--\eqref{HB_R}.

\begin{lemma}
For any $b\in[0,1]$ we have 
\begin{equation}
  H^\pm(b,b,b;\DT,\DX) = b\,,  
\end{equation}
\begin{align}\label{restoL}
    & H^{AB}_L(b,b,b;\DT,\DX,\alpha)=b - \frac{\DT}{\Delta_L+\alpha \DT} \RRR_L(b,\alpha) {\blue \le b}, \\ 
    \nonumber 
    &H^{C}_L(b,b,b,b;\DT,\DX, \Delta_L,\alpha)=b - \frac{\DT}{\Delta_L+\DX +\alpha \DT } \RRR_L(b,\alpha),
    \end{align}
    \begin{align}\label{eq:restoR}
     &H^{AC}_R(b,b,b;\DT,\DX,\alpha) = b - \frac{\DT}{\Delta_R-\alpha \DT} \RRR_R(b,\alpha),\\
     & H^{B}_R(b,b,b,b;\DT,\DX, \Delta_R,\alpha) =b - \frac{\DT}{\Delta_R+\DX-\alpha \DT} \RRR_R(b,\alpha).
\end{align}
Moreover, if $b$ and $c$ satisfy the Rankine-Hugoniot condition, namely  $f(b)+f(c) = \alpha (c-b)$, then 
\begin{align*}
   &H^{AB}_L (b,b,c;\DT,\Delta_L, \alpha) =  H^C_L (b,b,b,c;\DT,\DX, \Delta_L, \alpha) = b\,,
     \\[2mm] 
    &H^{AC}_R (b,c,c;\DT,\Delta_R, \alpha) =H^B_R (b,c,c,c;\DT,\DX, \Delta_R, \alpha)= c\,.
\end{align*}
\end{lemma}

\begin{proof}
These properties are direct consequences of Lemma~\ref{rem:mon-fluxes} and of the identities
\begin{equation*} 
h_0(b,b,\alpha) - h^-(b,b) = -\alpha b + \RRR_L(b,\alpha)\,,\quad h^+(b,b) - h_0(b,b,\alpha) = -\alpha b + \RRR_R(b,\alpha)\,.
\end{equation*}
\end{proof}

\begin{proposition}{\bf (Uniform $L^\infty$ bounds).}
Let $\rho_{\max}=|\rho_0\|_\infty$. For any $(t,x)\in[0,T]\times[-1,1]$, the numerical solution $\rho^\DX$ satisfies the bounds
 \begin{equation}\label{Linfty-bound}
   0\le  \rho^\DX(x,t) \le \|\rho_0\|_\infty.
 \end{equation}
\end{proposition}
\begin{proof}
At $t=0$, the property \eqref{Linfty-bound} holds by the definition \eqref{eq:rho-0-j} of the $\rho^0_j$. If the property holds up to $t=n\DT $ then all the densities $\rho^{n+1}_j$, for $-2^{N}\le j\le 2^{N}$, $j\neq m(n+1)$, are the values of one of the functions $H^\pm$, $H^{AB}_L$, $H^{AC}_R$, $H^B_R$ and $H^C_L$. Then, in Case \textbf{B} we have
\begin{align*}
 \underbrace{H^+(0,0,0 ; \DT, \DX)}_{=0} \leq \underbrace{H^+(\rho^{n}_{j-1},\rho^{n}_{j},\rho^{n}_{j+1} ;  \DT, \DX)}_{=\rho^{n+1}_j} \leq \underbrace{H^+(\rho_{\max},\rho_{\max},\rho_{\max}; \DT, \DX)}_{=\rho_{\max}},
\end{align*}
\begin{align*}
 \underbrace{H^{AB}_L(0,0,0 ;\DT, \DX, \alpha_{n,\DT})}_{=0}&\leq \underbrace{H^{AB}_L(\rho^{n}_{m(n)-2},\rho^{n}_{L},\rho^{n}_{R} ;  \DT, \DX,\alpha_{n,\DT})}_{=\rho^{n+1}_L}\\
 &\leq \underbrace{H^{AB}_L(\rho_{\max},\rho_{\max},\rho_{\max}; \DT, \DX,\alpha_{n,\DT})}_{=\rho_{\max}-\frac{\DT}{\Delta_L^{n+1}} \RRR_L(b,\alpha)}\leq \rho_{\max},
\end{align*}
\begin{align*}
 \underbrace{H^{B}_R(0,0,0,0 ; \DT, \DX, \Delta_R, \alpha_{n,\DT})}_{=0} &\leq \underbrace{H^{B}_R(\rho^{n}_{L},\rho^{n}_{R}, \rho^{n}_{m(n)+2},\rho^{n}_{m(n)+3} ;  \DT, \DX,\Delta_R,\alpha_{n,\DT})}_{=\rho^{n+1}_R} \\
 &\leq \underbrace{H^{B}_R(\rho_{\max},\rho_{\max},\rho_{\max},\rho_{\max}; \DT, \DX,\Delta_R, \alpha_{n,\DT})}_{=\rho_{\max}-\frac{\DT}{\Delta_R^{n+1}} \RRR_R(b,\alpha)}\leq \rho_{\max}.
\end{align*}
Cases \textbf{A} and \textbf{C} are similar. Therefore, the inequalities in \eqref{Linfty-bound} hold by induction. 
\end{proof}

\subsection{Approximate entropy inequalities}
In the following lemma, we introduce a discrete version of entropy inequalities. As in a large part of the literature, we consider the family of Kruzhkov entropies $\eta_k(\rho) = |\rho - k|$, for $k\in [0,1]$ and introduce suitable discrete entropy fluxes. In particular, we write
\begin{align*}
    &Q^{\pm,n}_j = h^\pm(\rho^n_{j-1}\vee k, \rho^n_{j}\vee k) - h^\pm(\rho^n_{j-1}\wedge k, \rho^n_{j}\wedge k),\\
    &Q_0^n = h_0(\rho^n_L\vee k, \rho^n_R\vee k,\alpha_{n,\DT}) - h_0(\rho^n_L\wedge k, \rho^n_R\wedge k,\alpha_{n,\DT})\,.
\end{align*}
In the next lemma, we deduce entropy inequalities valid for each cell.

\begin{lemma} \textbf{(Discrete entropy inequalities).}

Fix $k\in [0,1]$, $n\in\mathbb{N}$ and $j \in \left\{ -\frac 1\DX, \ldots , \frac 1 \DX \right\}$ with $j\not = m(n+1)$. Then the following holds.
 
\boxed{\textbf{Case \textbf{A}}}:
If $j\le m(n+1)-2$ and $j\ge m(n+1)+2$, then
 \begin{equation}\label{eq:discr-entropy-normal}
     \DX|\rho^{n+1}_j - k|\leq \DX|\rho^{n}_j - k| - \DT (Q^{\pm,n}_{j+1} - Q^{\pm,n}_j) 
 \end{equation}
 and 
 \begin{equation}\label{eq:discr-entropy-leftAB}
     \Delta_L^{n+1}|\rho^{n+1}_L - k|\leq \Delta_L^n|\rho^{n}_L - k| - \DT (Q^n_{0} - Q^{-,n}_{m(n)-1}) + \DT\RRR_L(k,\alpha_{n,\DT}), 
 \end{equation}
 \begin{equation}\label{eq:discr-entropy-rightAC}
     \Delta_R^{n+1}|\rho^{n+1}_R - k|\leq \Delta_R^n|\rho^{n}_R - k| - \DT (Q^{+,n}_{m(n){\color{purple}+2}} -Q^n_{0} )+ \DT\RRR_R(k,\alpha_{n,\DT})\,.
 \end{equation}

\boxed{\textbf{Case \textbf{B}}}:
 The inequality \eqref{eq:discr-entropy-normal} holds for $j\le m(n+1) -3$ and $j\ge m(n+1) + 2$. 
 Moreover, the following analogous of inequalities \eqref{eq:discr-entropy-leftAB}, \eqref{eq:discr-entropy-rightAC} are satisfied:
 \begin{equation}\label{eq:discr-entropy-leftB}
     \left(\Delta_L^{n+1} + \DX\right)|\bar \rho - k|\leq \Delta_L^n|\rho^{n}_L - k| - \DT (Q^n_{0} - Q^{-,n}_{m(n)-1}) + \DT\RRR_L(k,\alpha_{n,\DT}),
 \end{equation} 
 and 
 \begin{align}\label{eq:discr-entropy-rightB}
     \Delta_R^{n+1}|\rho^{n+1}_R - k|\leq & \,\Delta_R^n|\rho^{n}_R - k| + \DX |\rho^{n}_{m(n)+2} - k|\\
     & - \DT (Q^{+,n}_{m(n)+3} -Q^n_{0} )+ \DT\RRR_R(k,\alpha_{n,\DT}) \nonumber
 \end{align}
where $\bar \rho$ is defined at \eqref{eq:march_caseB_L}.

 \boxed{\textbf{Case \textbf{C}}}:
 The inequality \eqref{eq:discr-entropy-normal} holds for $j\le m(n+1) -2$ and $j\ge m(n+1) + 3$. 
 Moreover, the following analogous of inequalities \eqref{eq:discr-entropy-leftAB}, \eqref{eq:discr-entropy-rightAC} are satisfied:
  \begin{equation}\label{eq:discr-entropy-rightC}
     (\Delta_R^{n+1}+ \DX)|\bar\rho - k|\leq \Delta_R^n|\rho^{n}_R - k| - \DT (Q^{+,n}_{m(n)+2} -Q^n_{0} )+ \DT\RRR_R(k,\alpha_{n,\DT}), 
 \end{equation}
 and 
 \begin{align}\label{eq:discr-entropy-leftC}
     \Delta_L^{n+1}|\rho^{n+1}_L - k|\leq & \, \Delta_L^n|\rho^{n}_L - k|+\DX|\rho^{n}_{m(n)-2} - k| \\
     & - \DT (Q^n_{0} - Q^{-,n}_{m(n)-2}) + \DT\RRR_L(k,\alpha_{n,\DT}) \nonumber
 \end{align}
where $\bar \rho$ is defined at \eqref{def:right-values-caseC}.
\end{lemma}
 \begin{proof}
 The inequality in \eqref{eq:discr-entropy-normal} is classical {\blue and it is based on \eqref{eq:standard-marching-minus}, \eqref{eq:standard-marching-plus}}. 
 
 The computations in the proofs of \eqref{eq:discr-entropy-leftAB} - \eqref{eq:discr-entropy-leftC} being very similar, we only detail the inequalities involved in Case \textbf{B} to get \eqref{eq:discr-entropy-leftB} and \eqref{eq:discr-entropy-rightB}. Remember that in this case  $\rho^{n+1}_{m(n+1)-2}$ and $\rho^{n+1}_L$ coincide, see \eqref{def:left-values-caseB}, and their common value $\bar\rho$ is computed via \eqref{eq:march_caseB_L}.
 
 Due to the monotonicity properties of the function $H^{AB}_L$ and recalling \eqref{restoL}, 
 we have that  
   \begin{align*}
   &H^{AB}_L(\rho^n_{m(n)-2}\wedge k,\rho^n_L\wedge k,\rho^n_R\wedge k ; \DT, \Delta^n_L, \alpha_{n,\DT})\\
   &\quad \le 
     H^{AB}_L(\rho^n_{m(n)-2},\rho^n_L,\rho^n_R; \DT, \Delta^n_L, \alpha_{n,\DT}) \wedge H^{AB}_L(k,k,k; \DT, \Delta^n_L, \alpha_{n,\DT}) \\
     &\quad \le \bar\rho\wedge k
       \end{align*}
       and similarly that
      \begin{align*}     
  \bar\rho\vee k &\le H^{AB}_L(\rho^n_{m(n)-2}\vee k,\rho^n_L\vee k,\rho^n_R\vee k; \DT, \Delta^n_L, \alpha_{n,\DT})  + \frac{\DT}{\Delta_L^{n}+ \alpha_{n,\DT}\DT} \RRR_L(k,\alpha_{n,\DT}) \,. 
  \end{align*}

Recalling \eqref{caseB-id_DX},
\begin{equation*}
\DX + \Delta^{n+1}_L= \Delta^n_L +\alpha_{n,\DT} \DT\,,\qquad 
\Delta^{n+1}_R= \Delta^n_R -\alpha_{n,\DT} \DT + \DX\,,    \end{equation*}
and the definition of $H^{AB}_L$ in \eqref{HA_L}, 
we combine the previous inequalities
to find that
  \begin{align*}
   & \left(\DX+ \Delta^{n+1}_L \right)|\bar\rho- k| =\left(\DX+ \Delta^{n+1}_L \right)\left[(\bar\rho\vee k)- (\bar\rho\wedge k)\right]\\
    &\qquad \le \left(\DX+ \Delta^{n+1}_L \right) \left[ H^{AB}_L(\rho^n_{m(n)-2}\vee k,\rho^n_L\vee k,\rho^n_R\vee k; \DT, \Delta^n_L, \alpha_{n,\DT})\right.\\ &\qquad \quad \left. -H^{AB}_L(\rho^n_{m(n)-2}\wedge k,\rho^n_L\wedge k,\rho^n_R\wedge k; \DT, \Delta^n_L, \alpha_{n,\DT}) + \frac{\DT}{\Delta_L^{n}+ \alpha_{n,\DT}\DT} \RRR_L(k,\alpha_{n,\DT})\right] \\
    &\qquad = \Delta^n_L \left[(\rho^{n}_{L}\vee k)- (\rho^{n}_{L}\wedge k)\right]\\
    &\qquad\quad -\DT \Big[ h_0(\rho^n_L\vee k,\rho^n_R\vee k,\alpha_{n,\DT})-h_0(\rho^n_L\wedge k,\rho^n_R\wedge k, \alpha_{n,\DT})\Big.\\
    &\qquad \quad-\left.h^-(\rho^n_{m(n)-2}\vee k,\rho^n_L\vee k)+h^-(\rho^n_{m(n)-2}\wedge k,\rho^n_L\wedge k)\right] + \DT \RRR_L(k,\alpha_{n,\DT})\\
    &\qquad = \Delta^{n}_L |\rho^{n}_{L}- k| - \DT(Q^n_{0} - Q^{-,n}_{m(n)-1})+ \DT \,\RRR_L(k,\alpha_{n,\DT}).
  \end{align*}
  This proves \eqref{eq:discr-entropy-leftAB}. In a similar way 
  \begin{align*}
     &\rho^{n+1}_{R}\wedge k \ge H^{B}_R(\rho^n_L\wedge k,\rho^n_R\wedge k,\rho^n_{m(n)+2}\wedge k,\rho^n_{m(n)+3}\wedge k ; \DT,\DX, \Delta^n_R, \alpha_{n,\DT}),\\
     &\rho^{n+1}_{R}\vee k \le H^{B}_R(\rho^n_L\vee k,\rho^n_R\vee k, \rho^n_{m(n)+2}\vee k,\rho^n_{m(n)+3}\vee k ;\DT,\DX, \Delta^n_R, \alpha_{n,\DT}) \\&\phantom{\Delta^{n+1}_R(\rho^{n+1}_{R}\vee k)}\qquad + \frac{\DT}{\Delta_R^n+\DX-\alpha_{n,\DT}\DT} \RRR_R(k,\alpha_{n,\DT}) , 
  \end{align*}
  then 
  \begin{align*}
    \Delta^{n+1}_R &|\rho^{n+1}_{R}- k| =\Delta^{n+1}_R\left[(\rho^{n+1}_{R}\vee k)- (\rho^{n+1}_{R}\wedge k)\right]\\\le & ~\Delta^{n+1}_R\left[H^{B}_R(\rho^n_L\vee k,\rho^n_R\vee k, \rho^n_{m(n)+2}\vee k,\rho^n_{m(n)+3}\vee k ;\DT,\DX, \Delta^n_R, \alpha_{n,\DT})\right. \\ &\quad -H^{B}_R(\rho^n_L\wedge k,\rho^n_R\wedge k,\rho^n_{m(n)+2}\wedge k,\rho^n_{m(n)+3}\wedge k ; \DT,\DX, \Delta^n_R, \alpha_{n,\DT})\\ &\left. \quad + \frac{\DT}{\Delta_R^n+\DX-\alpha_{n,\DT}\DT} \RRR_R(k,\alpha_{n,\DT})\right] \\=
    &~ \Delta^n_R \left[(\rho^{n}_{R}\vee k)- (\rho^{n}_{R}\wedge k)\right] +\DX \left[(\rho^{n}_{m(n)+2}\vee k)- (\rho^{n}_{m(n)+2}\wedge k)\right]\\
    &\quad - \DT \Big[ h^+(\rho^n_{m(n)+2}\vee k,\rho^n_{m(n)+3}\vee k)-h^+(\rho^n_{m(n)+2}\wedge k,\rho^n_{m(n)+3}\wedge k)\Big.\\
    &\quad -\left.h_0(\rho^n_L\vee k,\rho^n_R\vee k,\alpha_{n,\DT})+h_0(\rho^n_L\wedge k,\rho^n_R\wedge k, \alpha_{n,\DT})\right] + \DT \RRR_R(k,\alpha_{n,\DT})\\
    =&~ \Delta^{n}_R |\rho^{n}_{R}- k|+ \DX |\rho^{n}_{m(n)+2}-k| - \DT( Q^{+,n}_{m(n)+3} -Q^n_{0} )+ \DT \RRR_R(k,\alpha_{n,\DT}).
  \end{align*}
  This completes the proof of \eqref{eq:discr-entropy-rightAC}.
 \end{proof}

Next, we define the approximate entropy flux $Q^\DX(x,t)$  
as follows: 
\begin{align*}
    &Q^\DX (x,t) 
    \\ &=\sum_n \left[
    \sum_{j\le m(n)-1 
    } Q^{-,n}_j  \cdot \mathbbm{1}_{ P^n_{j-1}(x,t)} + Q_0^n \cdot \mathbbm{1}_{P^n_{m(n)-1}}(x,t) + \sum_{j\ge m(n)+1
    } Q^{+,n}_{j+1}  \cdot \mathbbm{1}_{P^n_j}(x,t)\right]\,,
\end{align*}
where, using \eqref{eq:def-alpha_n} and Definition~\ref{def:the-mesh}, we set 
\begin{align*}
P^n_{j} & : = C_{j} \times  [t^n, t^{n+1})\,,\qquad j\not = m(n), m(n)\pm 1\\
P^n_L & : = \left\{(x,t): x_{m(n)-3/2}\le x < \xi^\DT(t)\,,\ t^n\le t < t^{n+1} \right\}\\
P^n_R & : = \left\{(x,t): \xi^\DT(t) < x \le x_{m(n)+3/2} \,,\ t^n\le t < t^{n+1} \right\}\,.
\end{align*}
%
Otherwise written as
\begin{equation*}
    Q^\DX (x,t) = \begin{cases}
        Q^{-,n}_{j+1} & (x,t) \in P^n_j, \qquad j \le m(n)-2\\
         Q_0^n & (x,t)\in P^n_L 
         \\
         Q^{+,n}_{j+1} & (x,t) \in P^n_j, \qquad j \ge m(n)+1
    \end{cases}
\end{equation*}

\begin{proposition}{\textbf{(Approximate entropy inequalities).}} Fix $n\in\mathbb{N}$, $k\in [0,1]$ and let $\varphi\in C^\infty_c([0,T]\times\eR)$, with $\varphi \ge 0$. There holds 
\begin{align*}
    &\int_{t^n}^{t^{n+1}}\int_\eR \left( |\rho^\DX - k| \partial_t\varphi + Q^\Delta(\rho^\DX, k)\partial_x\varphi\right)\, dx\,dt + \\
    &+ \int_\eR |\rho^\DX (t^n, x) - k| \varphi(t^n, x)\, dx - \int_\eR |\rho^\DX (t^{n+1}, x) - k| \varphi(t^{n+1}, x)\, dx \\
    &+ \int_{t^n}^{t^{n+1}}2f(k) \varphi(t, \xi^{\DT}(t))\, dx\geq \mathcal{O}(\DX^2) +\mathcal{O}(\DX\DT)+\mathcal{O}(\DT^2).
\end{align*}
\end{proposition}

\begin{proof}
For any $n\in \mathbb{N}$, we can define a piecewise constant discretization of the profile of test function $\varphi$ at time $t= t^n$
\begin{equation*}
 \varphi(t^n, x) \approx \sum_{j\neq m(n)}   \varphi_j^n \mathbbm{1}_{\CCC^n_j}(x) = \sum_{j\neq m(n)}  \mathbbm{1}_{\CCC^n_j}(x)  \frac{1}{|\CCC^n_j|}\int_{\CCC^n_j} \varphi(t^n, x) \, dx.
\end{equation*}
Recalling the notation \eqref{def:cell-size}  
for $\Delta^n_L$ and $\Delta^n_R$,
we consider 
 \begin{align*}
    &\varphi_{\Delta^n_R}^{n+1}  =  \frac{1}{\Delta^n_R}\int_{D^n_R} \varphi(t^{n+1}, x) \, dx,\qquad \hfill \varphi_{\Delta^n_L}^{n+1}  =  \frac{1}{\Delta^n_L}\int_{D^n_L} \varphi(t^{n+1}, x) \, dx,\\
  &\varphi_{\Delta\xi}^{n+1}  = \frac{1}{\alpha_{n,\DT} \DT} \int^{\xi^{n+1}}_{\xi^n}\fhi(t^{n+1}, x) \, d x  \quad\hfill
  \varphi_{m(n)+2}^{n+1}  = \frac{1}{\DX} \int^{x_{m(n)+5/2}}_{x_{m(n)+3/2}}\fhi(t^{n+1}, x) \, d x  .
   \end{align*}
 %
 The following computations apply to Case B, that is, $m(n+1) = m(n) +1$.
%
 We multiply each of the discrete entropy inequalities in  \eqref{eq:discr-entropy-normal}, \eqref{eq:discr-entropy-leftB} and \eqref{eq:discr-entropy-rightB} by $\fhi^{n+1}_j$, for $j\neq m(n+1)$ and we sum them to obtain
 \begin{align}
   & \sum_{j\neq m(n+1)} |\CCC^{n+1}_j| |\rho^{n+1}_j - k| \fhi^{n+1}_j \nonumber\\
    &\leq \sum_{j\leq m(n+1)-3}\left[ |\CCC^{n}_j| |\rho^{n}_j - k| - \DT \left( Q^{-, n}_{j+1} - Q^{-, n}_{j}\right)\right] \fhi^{n+1}_j \nonumber\\
    &\  + \sum_{j\geq m(n+1)-2}\left[ |\CCC^{n}_j| |\rho^{n}_j - k| - \DT \left( Q^{+, n}_{j+1} - Q^{+, n}_{j}\right)\right] \fhi^{n+1}_j \label{approx-entropy1} \\
    & \qquad\  + \left[\Delta^n_L |\rho^n_L-k| - \DT \left( Q^n_0 - Q^{-, n}_{m(n)-1}\right) + \DT \RRR_L(k,\alpha_{n,\DT})\right] \fhi^{n+1}_{m(n+1)-1} \nonumber\\
    &\   +\DX |\bar\rho-k| \left( \fhi^{n+1}_{m(n+1)-2}-\fhi^{n+1}_{m(n+1)-1}\right) \nonumber\\
    &\  + \left[\Delta^n_R |\rho^n_R-k| + \DX \nonumber |\rho^n_{m(n)+2}-k|- \DT \left( Q^{+, n}_{m(n)+3} -Q^n_0 \right) + \DT \RRR_R(k,\alpha_{n,\DT})\right] \fhi^{n+1}_{m(n+1)+1}.
 \end{align}
  Thanks to equality \eqref{caseB-id_DX} and the fact that $|\Delta^p_L| + |\Delta^p_R| = 3\DX$, for any $p$, we observe that  
  \begin{align*}
    &\Delta^n_R |\rho^n_R-k|  \fhi^{n+1}_{m(n+1)+1} = \Delta^n_R |\rho^n_R-k| \frac{1}{\Delta^{n+1}_R} \int_{\Delta^{n+1}_R} \fhi(t^{n+1}, x) \, d x \\
    & = \Delta^n_R |\rho^n_R-k| \frac{1}{\Delta^{n+1}_R} \left[ |\Delta^n_R| \frac{1}{|\Delta^{n}_R|} \int_{\Delta^{n}_R} \fhi(t^{n+1}, x) \, d x  +\DX \frac{1}{\DX} \int^{x_{m(n+1)+3/2}}_{x_{m(n+1)+1/2}} \fhi(t^{n+1}, x) \, dx \right.\\
    &\phantom{|\Delta^n_R| |\rho^n_R-k|}\left.-\alpha_{n,\DT} \DT \frac{1}{\alpha_{n,\DT} \DT} \int^{\xi^{n+1}}_{\xi^n}\fhi(t^{n+1}, x) \, d x   \right]\\
    & = \Delta^n_R |\rho^n_R-k| \frac{1}{\Delta^{n+1}_R} \left[ |\Delta^n_R| \fhi^{n+1}_{\Delta^{n}_R}  +\DX \fhi^{n+1}_{m(n)+2}  -\alpha_{n,\DT} \DT  \varphi_{\Delta\xi}^{n+1} \right]\\
    & = \Delta^n_R |\rho^n_R-k| \fhi^{n+1}_{\Delta^{n}_R}   + \Delta^n_R |\rho^n_R-k| \left[ - \frac{\Delta^{n+1}_R- |\Delta^{n}_R|}{\Delta^{n+1}_R} \fhi^{n+1}_{\Delta^{n}_R}+\frac{\DX}{\Delta^{n+1}_R}  \fhi^{n+1}_{m(n)+2}  -\frac{\alpha_{n,\DT} \DT }{\Delta^{n+1}_R}  \varphi_{\Delta\xi}^{n+1} \right]\\
    &= \Delta^n_R |\rho^n_R-k| \fhi^{n+1}_{\Delta^{n}_R}   + \Delta^n_R |\rho^n_R-k| \left[  \frac{\DX}{\Delta^{n+1}_R} \left( \fhi^{n+1}_{m(n)+2} -\fhi^{n+1}_{\Delta^{n}_R}\right)  -\frac{\alpha_{n,\DT} \DT }{\Delta^{n+1}_R} \left( \varphi_{\Delta\xi}^{n+1}-\fhi^{n+1}_{\Delta^{n}_R}\right) \right]\\
    &= \Delta^n_R |\rho^n_R-k| \fhi^{n+1}_{\Delta^{n}_R} + \left(\mathcal{O}(\DX^2) + \mathcal{O}(\DX\DT)\right)\|\rho_0\|_\infty \|\partial_x\fhi\|_\infty ,
  \end{align*}
  and, recalling that $|C^n_{m(n)+2}|=\DX$
  \begin{align*}
   & |C^n_{m(n)+2}| |\rho^n_{m(n)+2}-k|  \fhi^{n+1}_{m(n+1)+1} = |C^n_{m(n)+2}| |\rho^n_{m(n)+2}-k| \fhi^{n+1}_{m(n)+2}\\
  &\qquad \qquad+|C^n_{m(n)-2}| |\rho^n_{m(n)+2}-k|  \left[  \frac{|\Delta^{n}_R|}{\Delta^{n+1}_R} \left( \fhi^{n+1}_{m(n)+2} -\fhi^{n+1}_{\Delta^{n}_R}\right)  -\frac{\alpha_{n,\DT} \DT }{\Delta^{n+1}_R} \left( \fhi^{n+1}_{m(n)+2}-\varphi_{\Delta\xi}^{n+1}\right) \right] \\
  &\qquad= |C^n_{m(n)+2}| |\rho^n_{m(n)+2}-k| \fhi^{n+1}_{m(n)+2} + \left(\mathcal{O}(\DX^2) + \mathcal{O}(\DX\DT)\right)\|\rho_0\|_\infty \|\partial_x\fhi\|_\infty ,
  \end{align*}
  On the left side of the interface, we similarly have, 
   \begin{align*}
    &\Delta^n_L |\rho^n_L-k|  \fhi^{n+1}_{m(n+1)-1} = \Delta^n_L |\rho^n_L-k| \frac{1}{\Delta^{n+1}_L} \int_{\Delta^{n+1}_R} \fhi(t^{n+1}, x) \, d x \\
    & = \Delta^n_L |\rho^n_L-k| \frac{1}{\Delta^{n+1}_L} \left[ |\Delta^n_L| \fhi^{n+1}_{\Delta^{n}_L}  -\DX \fhi^{n+1}_{m(n+1)-2}  +\alpha_{n,\DT} \DT  \varphi_{\Delta\xi}^{n+1} \right]\\
    &= \Delta^n_L |\rho^n_L-k| \fhi^{n+1}_{\Delta^{n}_L}   + \Delta^n_L |\rho^n_L-k| \left[  \frac{\DX}{\Delta^{n+1}_L} \left( \fhi^{n+1}_{m(n+1)-2} -\fhi^{n+1}_{\Delta^{n}_L}\right)  -\frac{\alpha_{n,\DT} \DT }{\Delta^{n+1}_L} \left( \varphi_{\Delta\xi}^{n+1}-\fhi^{n+1}_{\Delta^{n}_L}\right) \right]\\
    &= \Delta^n_L |\rho^n_L-k| \fhi^{n+1}_{\Delta^{n}_L} + \left(\mathcal{O}(\DX^2) + \mathcal{O}(\DX\DT)\right)\|\rho_0\|_\infty \|\partial_x\fhi\|_\infty .
  \end{align*}
  In the notation of Def.~\ref{def:the-mesh}, we recall that $\CCC^n_j = \CCC^{n+1}_j$ for all $j\leq m(n+1)-3 = m(n)-2$ or $j\geq m(n+1) +2=m(n)+3$ and that  $C^n_{m(n)-1} = D^n_L$, $C^n_{m(n)+1} = D^n_R$. With these notations, the estimates above and classical rearranging of the terms in \eqref{approx-entropy1} allow us to write 
 \begin{align}\label{approx-entropy2}
    \sum_{j\neq m(n+1)} &|\CCC^{n+1}_j| |\rho^{n+1}_j - k| \fhi^{n+1}_j- \sum_{j\neq m(n)-1, m(n), m(n)+1} |\CCC^{n}_j| |\rho^{n}_j - k| \fhi^{n+1}_j\nonumber\\
    &- |\Delta^n_L| |\rho^n_L-k|\fhi^{n+1}_{\Delta^n_L}-|\Delta^n_R| |\rho^n_R-k|\fhi^{n+1}_{\Delta^n_R}\nonumber\\
    & \leq -\DT \left[\sum_{j\leq m(n+1)-3}  Q^{-, n}_{j} \left(\fhi^{n+1}_{j-1}-\fhi^{n+1}_j\right)  
     + \sum_{j\geq m(n+1)+3} Q^{+, n}_{j} \left(\fhi^{n+1}_{j-1}-\fhi^{n+1}_j\right)\right]\nonumber\\
    &\qquad- \DT \left[ Q^{-, n}_{m(n)-1} \fhi^{n+1}_{m(n+1)-3}- Q^{+, n}_{m(n)+3}\fhi^{n+1}_{m(n+1)+2}\right.\nonumber\\
    &\qquad\left. +\left( Q^n_0 - Q^{-, n}_{m(n)-1}\right) \fhi^{n+1}_{m(n+1)-1} +\left( Q^{+, n}_{m(n)+3} -Q^n_0 \right)\fhi^{n+1}_{m(n+1)+1}\right]\nonumber\\
    &\qquad+ \DT\left[ \RRR_L(k,\alpha_{n,\DT}) \fhi^{n+1}_{m(n+1)-1}+\RRR_R(k,\alpha_{n,\DT}) \fhi^{n+1}_{m(n+1)+1}\right]\nonumber\\
    &\qquad +\DX |\bar\rho-k| \left( \fhi^{n+1}_{m(n+1)-2}-\fhi^{n+1}_{m(n+1)-1}\right)+\left(\mathcal{O}(\DX^2) + \mathcal{O}(\DX\DT)\right)\|\rho_0\|_\infty \|\partial_x\fhi\|_\infty\nonumber\\
    &\leq -\DT \left[\sum_{j\leq m(n+1)-2}  Q^{-, n}_{j} \left(\fhi^{n+1}_{j-1}-\fhi^{n+1}_j\right)  
     + \sum_{j\geq m(n+1)+2} Q^{+, n}_{j} \left(\fhi^{n+1}_{j-1}-\fhi^{n+1}_j\right)\right]\nonumber\\
    &\qquad- \DT \left[  Q^n_0 \left(\fhi^{n+1}_{m(n+1)-1}-\fhi^{n+1}_{m(n+1)+1}\right) + Q^{-, n}_{m(n)-1}\left(\fhi^{n+1}_{m(n+1)-2}-\fhi^{n+1}_{m(n+1)-1}\right)\right.\nonumber\\
    &\qquad\qquad \left. + Q^{+, n}_{m(n)+3} \left(\fhi^{n+1}_{m(n+1)+1}-\fhi^{n+1}_{m(n+1)+2}\right)\right]\nonumber\\
    &\qquad+ \DT\left[ \left(\RRR_L(k,\alpha_{n,\DT}) +\RRR_R(k,\alpha_{n,\DT})\right) \fhi^{n+1}_{m(n+1)-1}+\RRR_R(k,\alpha_{n,\DT}) \left(\fhi^{n+1}_{m(n+1)+1}-\fhi^{n+1}_{m(n+1)-1}\right)\right]\nonumber\\
    &\qquad +\DX |\bar\rho-k| \left( \fhi^{n+1}_{m(n+1)-2}-\fhi^{n+1}_{m(n+1)-1}\right)+\left(\mathcal{O}(\DX^2) + \mathcal{O}(\DX\DT)\right)\|\rho_0\|_\infty \|\partial_x\fhi\|_\infty .
 \end{align}
We can easily estimate some of the terms above
\begin{align}
 & \DT  Q^n_0 \left(\fhi^{n+1}_{m(n+1)-1}-\fhi^{n+1}_{m(n+1)+1}\right)\leq 8 \DT \DX \|f\|_\infty\| \|\partial_x \fhi\|_\infty ; \\
 & \DT Q^{-, n}_{m(n)-1} \left(\fhi^{n+1}_{m(n+1)-2}-\fhi^{n+1}_{m(n+1)-1}\right)\leq 4 \DT \DX \|f\|_\infty\| \|\partial_x \fhi\|_\infty ;\\
 & \DT Q^{+, n}_{m(n)+3} \left(\fhi^{n+1}_{m(n+1)+1}-\fhi^{n+1}_{m(n+1)+2}\right)\leq 4 \DT \DX \|f\|_\infty\| \|\partial_x \fhi\|_\infty ; \\
 &\DT\RRR_R(k,\alpha_{n,\DT}) \left(\fhi^{n+1}_{m(n+1)+1}-\fhi^{n+1}_{m(n+1)-1}\right) \leq 8 \DT \DX \|f\|_\infty\| \|\partial_x \fhi\|_\infty ; \\
 &\DX |\bar\rho-k| 
 \left( \fhi^{n+1}_{m(n+1)-2}-\fhi^{n+1}_{m(n+1)-1}\right)\leq 2 \DX^2 \|\rho_0\|_\infty\| \|\partial_x \fhi\|_\infty .
\end{align}
Recalling that $\RRR_L(k,\alpha_{n,\DT}) +\RRR_R(k,\alpha_{n,\DT}) = 2f(k)$, and using the estimates above we get  
\begin{align}
&\sum_{j\neq m(n+1)} |\CCC^{n+1}_j| |\rho^{n+1}_j - k| \fhi^{n+1}_j \label{line1}\\
&- \sum_{j\neq m(n), m(n)\pm 1}
|\CCC^{n}_j| |\rho^{n}_j - k| \fhi^{n+1}_j- \Delta^n_L |\rho^n_L-k|\fhi^{n+1}_{\Delta^n_L}- \Delta^n_R |\rho^n_R-k|\fhi^{n+1}_{\Delta^n_R}\label{line2}\\
& \leq -\DT \left[\sum_{j\leq m(n+1)-2}  Q^{-, n}_{j} \left(\fhi^{n+1}_{j-1}-\fhi^{n+1}_j\right)  + \sum_{j\geq m(n+1)+2} Q^{+, n}_{j} \left(\fhi^{n+1}_{j-1}-\fhi^{n+1}_j\right)\right]\label{line3}\\
&\qquad + 2 \DT f(k) \fhi^{n+1}_{m(n+1)-1}+\mathcal{O}(\DX^2) +\mathcal{O}(\DX\DT).\label{line4}
\end{align}
To which we add and subtract the terms 
\begin{align*}
     \sum_{j\neq m(n), m(n)\pm 1}
     |\CCC^{n}_j| |\rho^{n}_j - k| \fhi^{n}_j+ \Delta^n_L |\rho^n_L-k|\fhi^{n}_{\Delta^n_L}+ \Delta^n_R |\rho^n_R-k|\fhi^{n}_{\Delta^n_R},
\end{align*}
which can also be written as $\sum_{j\neq m(n)} |\CCC^{n}_j| |\rho^{n}_j - k| \fhi^{n}_j$, so that lines \eqref{line1}, \eqref{line2} become
\begin{align*}
  &\sum_{j\neq m(n+1)} |\CCC^{n+1}_j| |\rho^{n+1}_j - k| \fhi^{n+1}_j -\sum_{j\neq m(n)} |\CCC^{n}_j| |\rho^{n}_j - k| \fhi^{n}_j\\
&- \sum_{j\neq m(n)-1, m(n), m(n)+1} |\CCC^{n}_j| |\rho^{n}_j - k| (\fhi^{n+1}_j-\fhi^{n}_j)\\
&- \Delta^n_L |\rho^n_L-k|(\fhi^{n+1}_{\Delta^n_L}-\fhi^{n}_{\Delta^n_L})-\Delta^n_R |\rho^n_R-k|(\fhi^{n+1}_{\Delta^n_R}-\fhi^{n}_{\Delta^n_R}). 
\end{align*}

\end{proof}

\section{Numerical Examples and Validation}\label{sec:NumericalExamples}
We conclude the article by examining the performance and accuracy of the proposed numerical scheme applied to selected examples compared to the standard Godunov scheme without any moving mesh adaptation. 

\subsection{Examples}
To observe the changes in the slope of the interface and its interactions with classical waves, two examples are chosen.
  
Let $  f(\rho) = \rho(1-\rho)$ and consider the turning curve, 
\begin{equation*}
	\xi(t) := \xi_0 + \int_{0}^{t}\xi'(s)ds,
\end{equation*}
with $\xi_0=-0.1$ and the following sets of initial data and slope of the discontinuity :
\begin{center}
	\begin{tabular}{ll|l}
		\label{tab: test_problems}
		&Initial data & slope of the discontinuity 
		\\
		\hline
        \boxed{\textbf{A}}: & $\begin{aligned}
						\label{eq:square1}
						\rho_0(x) &= \begin{cases}
										3/5, & \text{if}~  -1< x < 3/10\\
										9/10, & \text{if}~  3/10 \leq x <  1\,.
									\end{cases}  
					\end{aligned}$ 
					& $\xi'(s) = \begin{cases}
						2/5, & \text{for }~ 0 < s < 3/5 \\
						0, & \text{for }~ s \geq 3/5 \\ 	
					\end{cases}$ \\ \hline
	 \boxed{\textbf{B}}: & $\begin{aligned}
						\label{eq:square4}
						\rho_0(x) &= \begin{cases}
										9/10, & \text{if}~  -1/10\leq x \leq 1/2\\
										3/5, & \text{if}~  \text{otherwise.}
									 \end{cases}
					\end{aligned}$
				   &
					$\begin{aligned} 
						\xi'(s) &= \begin{cases}
											1/10, & \text{for}~ 0 < s \leq 3/10, \\
											1/4, & \text{for}~ 3/10 < s \leq 3/5, 	\\
											0, & \text{for}~  s > 3/5.
										\end{cases}
					\end{aligned}$	\\	\hline 
	 \end{tabular}
\end{center}
 The exact solutions are constructed using the method of characteristics. As $\rho(0, x) > 1/2$, the boundary conditions require that rarefaction waves enter the boundary points $\{(t, -1^+)\}$ and $\{(t, 1^-)\}$ such that $\rho(t, \pm 1)\geq 1/2$ for all $t>0.$ 
\medskip

Example \boxed{\textbf{A}}: For $0<t\leq 4/9$, two undercompressive shocks, $\sigma_L(t)=-1/10(1+4t)$ and $\sigma_R(t)=-1/10(1-4t)$, separated by vacuum, are generated at $x=\xi(0)$, while at $x=0.3$, a single shock $\sigma_1(t)=0.3 -0.5t$ is also created. Observing that $\xi(t) = \sigma_R(t)$ and since $\dot{\sigma}_L > \dot{\sigma}_1$, the interface interacts with the shock $\sigma_1$ at $t=4/9$. The solution for $ t\in [0, 4/9)$ writes, 
\begin{equation*}
    \rho(t,x) = \begin{cases}
        0.6 & \text{if } -1 < x < \sigma_L(t)\\
        0 & \text{ if } -\sigma_L(t) \leq x < \xi(t)\\
        0.6 & \text{ if } \xi(t) \leq x < \sigma_R(t) \\
        0.9 & \text{ if }  \sigma_R(t) \leq x < 1.
    \end{cases}
\end{equation*}
At $t=4/9$, a slight decrease in density is observed and a solution to the Riemann problem at $\xi$ is calculated with an intermediate state $\rho_M = 1/10\left(7-1/2\sqrt{\frac{88}{10}}\right)$ for the time period between $4/9$ and $6/10$. The solution in this time frame is as follows:
\begin{equation*}
    \rho(t,x) = \begin{cases}
        0.6 & \text{if } -1 < x < \sigma_L(t),\\
        0 & \text{ if } -\sigma_L(t) \leq x < 7/90 - (t-4/9),\\
        \frac{x-7/90}{2(t-4/9)} + \frac{1}{2} & \text{ if } 7/90 - (t-4/9) \leq x < 7/90 - \dot{R}_{\max} (t-4/9), \\
        \rho_M & \text{ if }   7/90 - \dot{R}_{\max} (t-4/9) \leq x < \xi(t),\\
        0.9 & \text{ if }   \xi \leq x < 1.
    \end{cases}
\end{equation*}
Here $\dot{R}_{\max} = 1/5\left(2-1/2\sqrt{\frac{88}{10}}\right)$, denotes the maximum slope of the rarefaction fan. 
\medskip

Example \boxed{\textbf{B}}:
At $x=0.5$, a rarefaction is formed with a value of $1/2 +f'(\rho)t$, where $\rho$ is in the range of $[0.6, 0.9]$. Additionally, two shocks are created, $\sigma_L(t) = -0.1 + 0.4 t$ and $\sigma_R = -0.1 + 0.1 t$, which originate from $\xi(0) = -0.1$ for $0< t \leq 3/10$. The solution is as follows
\begin{equation*}
    \rho(t,x) = \begin{cases}
        6/10 & \text{ if } -1+8t/10 \leq x \leq \sigma_L(t),\\
        0 & \text{ if } \sigma_L(t) < x < \xi(t), \\
        0.9 & \text{ if } \xi(t) < x \leq R^0_{\min}(t), \\
        (t-x)/2t&\text{ if }  \dot{R}^0_{\min}(t) < (x-1/2)/t \leq \dot{R}^0_{\max}(t),\\
        6/10, &\text{ if } 1/2( 1- t/5) < x < 1-8t/10,
    \end{cases}
\end{equation*}
where $ \dot{R}^0_{\min}$ and $ \dot{R}^0_{\max}$ represents the minimum and maximum speeds of the $R^0.$ For $3/10\leq t \le 6/10$, the solution writes
\begin{equation*}
    \rho(t,x) = 
    \begin{cases}
        6/10 & \text{ if } -1+8t/10 \leq x \leq \sigma_L(t),\\
        0 & \text{ if } \sigma_L(t) < x < R^1_{\min}(t), \\
        \frac{1}{2} + \frac{x+1/40}{2(t-3/10)}, & \text{ if } \dot{R}^1_{\min}(t) \leq x \leq \dot{R}^1(t) \\
        \rM, & \text{ if } R^1_{\max}(t) \leq x < \xi(t) \\
        9/10, &\text{ if } \xi(t) \leq x < R^0_{\min}(t)\\
        (t-x)/2t&\text{ if }  \dot{R}^0_{\min}(t) < (x-1/2)/t \leq \dot{R}^0_{\max}(t),\\
        6/10 & R^0_{\max}(t) \leq x < 1-8t/10,
    \end{cases}
\end{equation*}
by noting that a Riemann problem at $\xi$ is solved with an intermediate state of $\rM = \frac{1}{80}(50 - \sqrt{1636})$ and a small rarefaction denoted $R^1(t) = -\frac{1}{40} - f'(\rho) \left(t-\frac{3}{10}\right)$ for $\rho\in [0, \rM]$ when it interacts with the shock $\sigma_R$. This occurs for $3/10\leq t < 6/10$.
The approximate solutions of these two examples are depicted in the $x-t$ can be found in Figure \ref{fig: color map example V_VI}. However, we exclusively display the simulation results of Example $\boxed{\textbf{A}}$ at selected times in Figure \ref{fig:eg_A_rho_x}. 
To illustrate the effectiveness of the scheme in handling discontinuities with negative slopes, particularly during wave interactions with $\xi$, we present in Figure \ref{fig:eg_C_rho_x} the simulation results of a third example, 
denoted as Example $\boxed{\textbf{C}}$, where the initial density $\rho_0$ is defined as in Example $\boxed{\textbf{A}}$. However, the slope of the flux interface is specified as follows:  
\begin{align} 
    \xi'(s) &= \begin{cases}
                        1/10, & \text{for}~ 0 < s \leq 3/10, \\
                        1/4, & \text{for}~ 3/10 < s \leq 6/10, 	\\
                        -3/2, & \text{for}~  6/10 < s \leq 1, \\ 
                        0, & \text{for}~ s>1.
                    \end{cases}
\end{align}

\begin{figure}[htbp]
	\centering
	\includegraphics[scale=0.45]{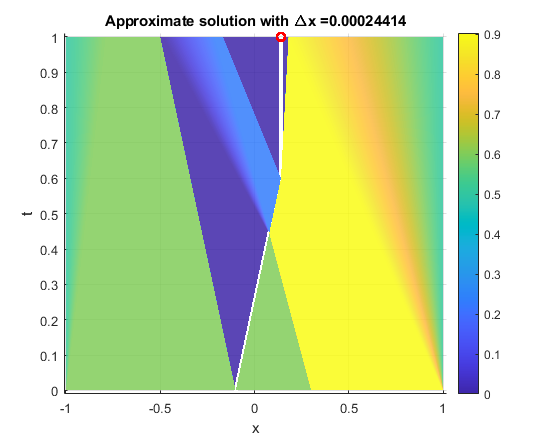} 
	\includegraphics[scale=0.44]{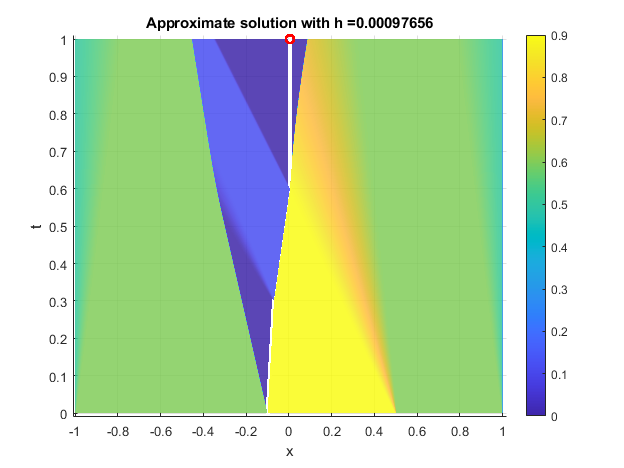}
	\caption{\small{Approximate solution of Examples \boxed{\textbf{A}} (left) and \boxed{\textbf{B}} (right) in the $x-t$ plane.}}
	\label{fig: color map example V_VI}
\end{figure}

\begin{center}
\begin{figure}
    \begin{tabular}{ll}
        \includegraphics[width=.5\linewidth,valign=m]{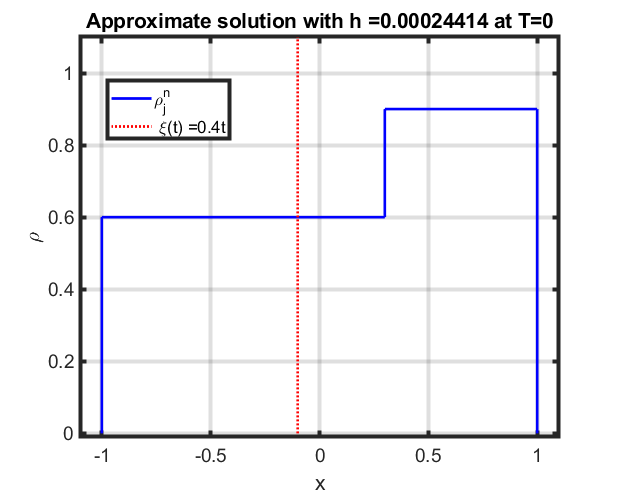} & \includegraphics[width=.5\linewidth,valign=m]{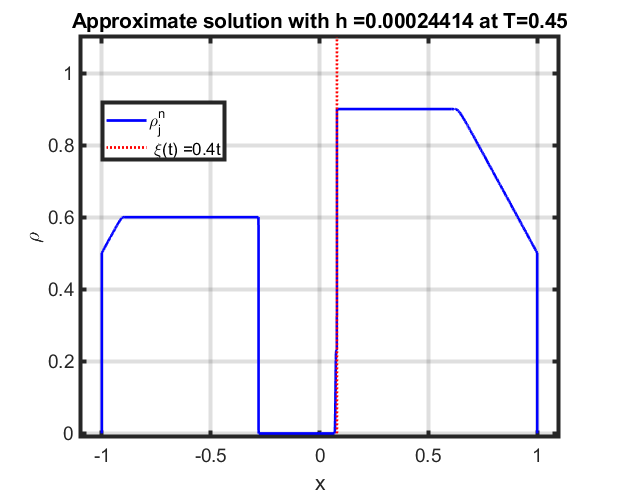}\\
        \includegraphics[width=.5\linewidth,valign=m]{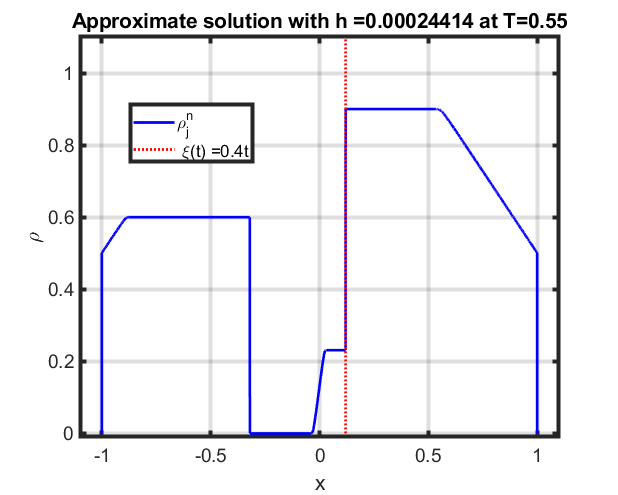} & \includegraphics[width=.5\linewidth,valign=m]{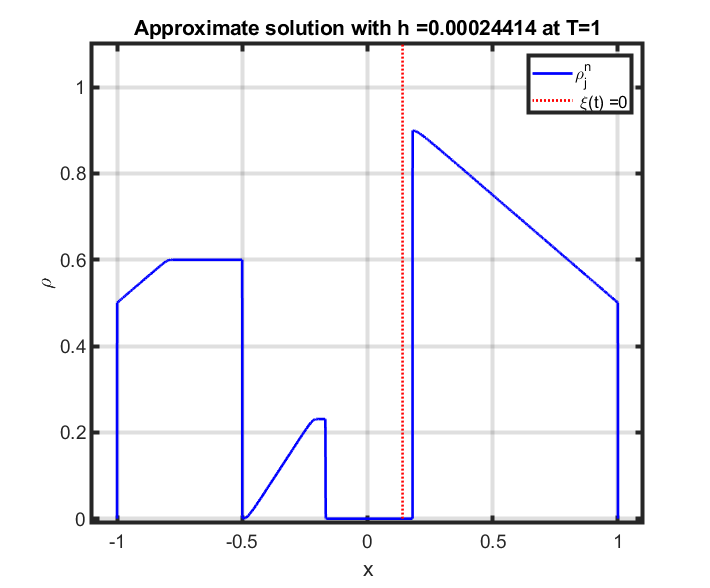}
      \end{tabular} 
      \caption{\small{Evolution of the approximate solution of Examples \boxed{\textbf{A}} in the $\rho-x$ plane for $t=0,~0.45, ~0.55, $ and $1.0$.}}
      \label{fig:eg_A_rho_x}
\end{figure} 
\end{center}

\begin{center}
    \begin{figure}
        \begin{tabular}{ll}
            \includegraphics[width=.5\linewidth,valign=m]{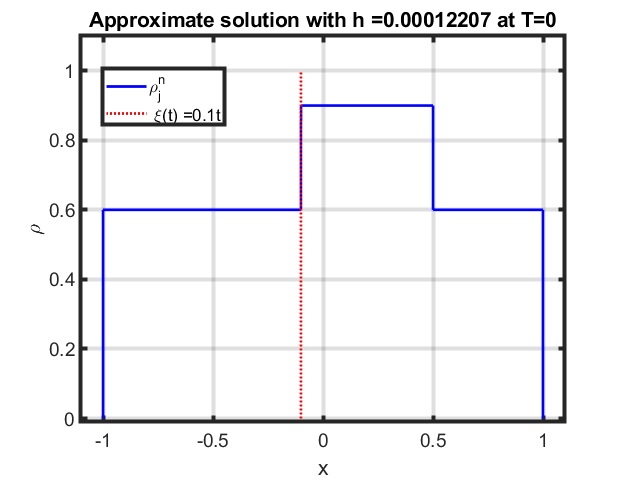} & \includegraphics[width=.5\linewidth,valign=m]{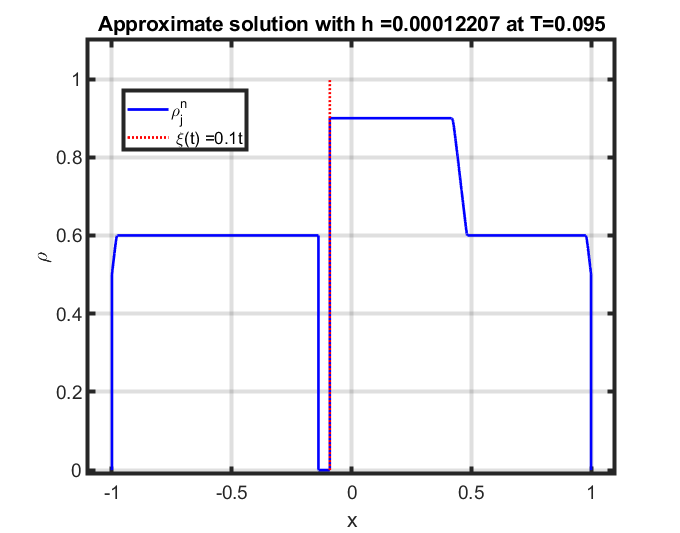}\\
            \includegraphics[width=.5\linewidth,valign=m]{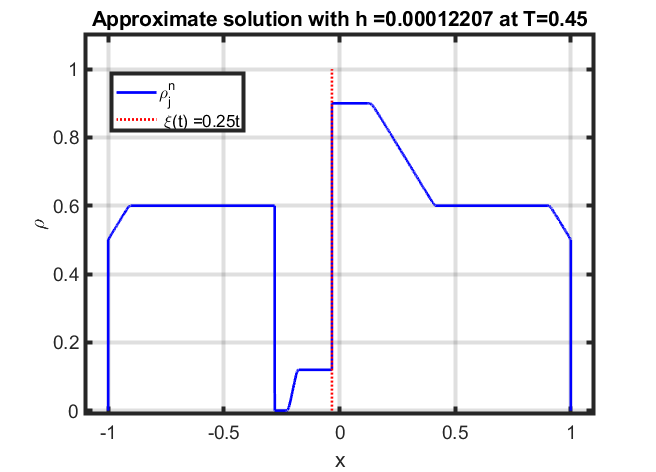} & \includegraphics[width=.5\linewidth,valign=m]{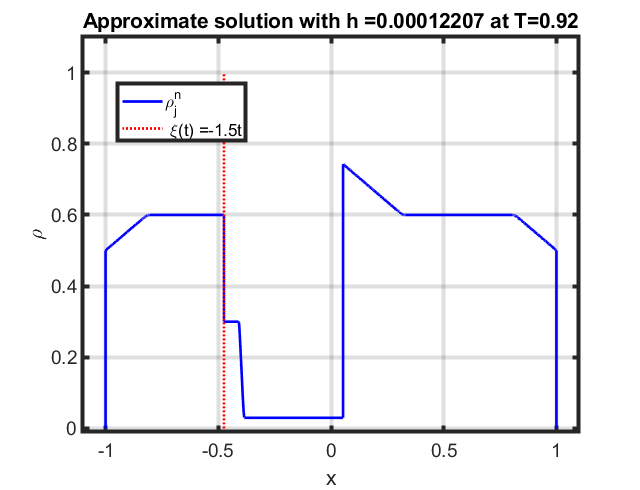}
          \end{tabular} 
          \caption{\small{Evolution of the approximate solution of Examples \boxed{\textbf{C}} in the $\rho-x$ plane for $t=0,~0.095, ~0.45, $ and $0.92$.}}
          \label{fig:eg_C_rho_x}
    \end{figure} 
    \end{center}

\subsection{Order of convergence}
In this section, we assess the accuracy of the numerical scheme we proposed by estimating the $\Lp{1}$ errors and deducing the order of convergence on Examples \boxed{\textbf{A}} and \boxed{\textbf{B}}. These examples are carefully chosen to enable us to observe the performance of our scheme during collisions between the turning curve and incoming waves. 
%
The numerical $\Lp{1}$-norm at time level $n = T/\DT$ writes 
\begin{equation}
    \label{eq:L1-error}
    Err(n,\DX) = \sum_{j\not = m(n)\,,\ |j|\le \frac 1\DX}\left|\rho(T, x^n_j)-\rho^{n}_j\right|\DX,
\end{equation}
where $\rho(T,x)$ is the exact solution which we use as the reference solution and $\rho_j^{n}$ is the approximate solution obtained with our scheme. In these simulations, we use a CFL number of 0.45, and plot the log-log graph of the errors against $N$ values for Examples A and B below at times where the interface collides with the incoming waves. We first show the graph of the approximate solution in the $x-t$ plane for examples A and B. 
\begin{figure}[htbp]
    \centering
	\includegraphics[scale=0.49]{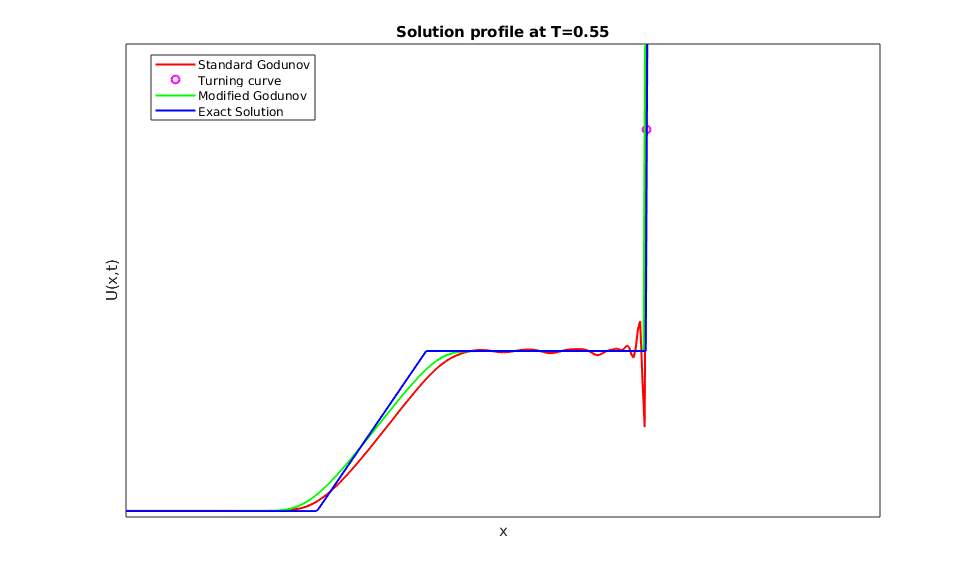}
	\caption{\small{A zoomed view of exact solution and the numerical solutions with the standard Godunov's flux and the modified flux at $T=0.55$. The violet circle represents the position of the interface.}}
	\label{fig:interaction-at-xi}
\end{figure}

We analyze the performance of our scheme with the standard Godunov flux by comparing Example A at $T=0.55$, where a collision is seen in figure \ref{fig:interaction-at-xi}. 
The figure shows that the standard Godunov flux produces spurious oscillations at $x=\xi(t)$ when a shock wave from the right interacts with $\xi$. In contrast, the modified scheme not only eliminates these oscillations, but also accurately captures the non-zero intermediate state that is formed after the interaction.
\begin{table}[htbp]
	\centering
\begin{tabular}{|c|c|c|c|c|}
	\hline
	& \multicolumn{2}{c|}{Example A} &\multicolumn{2}{c|}{Example B}\\
	\hline 
	$\DX$ & $Err(\DX)$ & $\frac{\ln \left(Err(\DX)\right)}{\ln\left(\DX\right)}$ & $Err(\DX)$ &$\frac{\ln \left(Err(\DX)\right)}{\ln\left(\DX\right)}$\\
	\hline
	$1.5625e-02$  & $3.4222e-02$ & -0.79 & $7.5633e-3$& -0.79 \\
	\hline
	$3.9062e-03$ & $1.0174e-02$ & -0.80 & $4.2297e-3$& -0.79 \\
	\hline
	$9.7656e-04 $ & $3.1818e-03$ & -0.81 & $3.1088e-3$& -0.79\\
	\hline
	$ 4.8828e-04$ & $1.7271e-03$ & -0.81 & $2.4516e-3$& -0.79\\
	\hline
	$2.4414e-04$ & $9.7091e-04$ & -0.81 & $1.8074e-3$& -0.79\\
	\hline
	$1.2207e-04$ & $5.2053e-04$ & -0.81 & $1.2818e-3$& -0.78\\
	\hline
\end{tabular}
\caption{\small \emph{$\Lp{1}-$errors at $T=0.55$, for example A and at $T=0.58$ for example B with mesh size from $\DX = 1/500, \ldots, 1/5000$.}}
\end{table}
\begin{figure}[htbp]
    \label{fig:Hughes_model_compare}
	\centering
	\includegraphics[scale = 0.52]{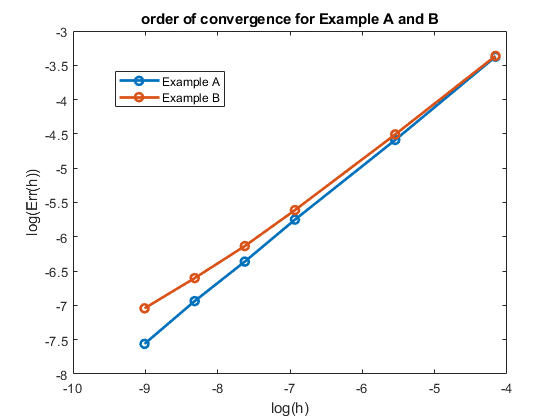}
	\caption{\small{The log-log graph of $\Lp{1}$ error against $\DX$ on Example A and B at $T=0.55$ and $T=0.58$ respectively.}}
\end{figure}

\clearpage


\printbibliography

\end{document}